\documentclass[12pt]{article}

\usepackage{amssymb,amsmath,amsthm,sectsty,url}
\usepackage[colorlinks,linkcolor=black,citecolor=black,filecolor=black,urlcolor=black]{hyperref}

\sectionfont{\large}
\subsectionfont{\normalsize}
\numberwithin{equation}{section}

\usepackage{a4wide}
\parskip \medskipamount

\newtheorem{maintheorem}{Theorem}
\newtheorem{theorem}{Theorem}[section]
\newtheorem{lemma}[theorem]{Lemma}
\newtheorem{proposition}[theorem]{Proposition}
\newtheorem{corollary}[theorem]{Corollary}
\newtheorem{fact}[theorem]{Fact}
\newtheorem{definition}[theorem]{Definition}
\newtheorem{remark}[theorem]{Remark}

\newtheorem{example}[theorem]{Example}
\newtheorem{open}[theorem]{Question}

\newtheorem*{lemma*}{Lemma}
\newtheorem*{proposition*}{Proposition}

\def\E{\mathop{\mathbb E}}

\newcommand{\PP}{\mathcal{P}}
\newcommand{\ek}{(2_k^{\mathrm{e}})}
\newcommand{\vk}{(2_k^{\mathrm{v}})}
\newcommand{\ekk}{(3_k^{\mathrm{e}})}
\newcommand{\vkk}{(3_k^{\mathrm{v}})}

\newcommand{\Pke}{ \mathcal{P}_k^{\mathrm{e}}}
\newcommand{\Pkv}{ \mathcal{P}_k^{\mathrm{v}}}
\newcommand{\sime}{\stackrel{\mathrm{e}}{\sim}}
\newcommand{\simv}{\stackrel{\mathrm{v}}{\sim}}

\renewcommand{\Pr}{ \mathrm P}

\newcommand{\rr}{\mathrm{r}}
\newcommand{\ee}{\mathrm{e}}
\newcommand{\vv}{\mathrm{v}}

\newcommand{\dE}{\overrightarrow{E}}

\newcommand{\de}{\overrightarrow{e} }

\newcommand{\df}{\overrightarrow{f} }

\newcommand{\dpi}{\overrightarrow{\pi} }

\newcommand{\Ind}[1]{\mathbf{1}\{#1\}}

\DeclareMathSymbol{\leqslant}{\mathalpha}{AMSa}{"36} 
\DeclareMathSymbol{\geqslant}{\mathalpha}{AMSa}{"3E} 
\DeclareMathSymbol{\eset}{\mathalpha}{AMSb}{"3F}     
\renewcommand{\leq}{\;\leqslant\;}                   

\renewcommand{\epsilon}{\varepsilon}

\newcommand{\sfrac}[2]{\mbox{\small $\frac{#1}{#2}$}}

\newcommand{\N}{\mathbb N}
\newcommand{\R}{\mathbb R}
\newcommand{\Z}{\mathbb Z}

\usepackage[normalem]{ulem}

\begin{document}

\title{Reversibility of the non-backtracking random walk}
\author{Jonathan Hermon
\thanks{University of Cambridge, Cambridge, UK. E-mail: {\tt jh2129@statslab.cam.ac.uk}. Financial support by
the EPSRC grant EP/L018896/1.}}

\date{}
\maketitle
 
\begin{abstract}
Let  $G$ be a connected graph   of uniformly bounded degree.  A $k$ non-backtracking random walk ($k$-NBRW) $(X_n)_{n =0}^{\infty}$ on $G$ evolves according to the following rule: Given $ (X_n)_{n =0}^{s}$, at time $s+1$ the walk picks at random some edge which is incident to $X_s$ that was not crossed in the last $k$ steps and moves to its other end-point. If no such edge exists then it makes a simple random walk step.  Assume that for some $R>0$ every ball of radius $R$ in $G$ contains a simple cycle of length at least $k$. We show that under some ``nice" random time change the $k$-NBRW becomes reversible. This is used to prove that it  is recurrent iff the simple random walk is recurrent.
\end{abstract}

\paragraph*{\bf Keywords:}
{\small Non-backtracking random walk, recurrence, transience.
}

\section{Introduction}
In this work we study a generalization of non-backtracking random walk (\textbf{NBRW}) on a graph $G$, in which the walk is required (when possible) to avoid either the last $k$ edges it crossed (viewed as undirected edges), or the last $k$ vertices it visited. If this is not possible, we say the walk got stuck, in which case it moves to a random neighbor. We call such a walk an edge (respectively, a vertex\footnote{When a certain statement is true for both the edge and the vertex NBRWs or if the type of walk is clear from context, we often omit the prefix edge/vertex.}) $k$\emph{-NBRW} (for more precise definitions see \S\ref{s:defofwalks}). 

The special case of $k=1$, which is simply a NBRW, received much attention (e.g.~\cite{alon,benhamu,beres,bordenave,CLT,kempton,lee,ram,GW,cogrowth2}). A recurring theme in the existing literature is that when $G$ is regular the NBRW and the simple random walk (\textbf{SRW}) on $G$ (see \S\ref{s:defofwalks} for the definition of SRW) are intimately related to one another.  In particular, when $G$ is regular one may deduce some properties of one walk by establishing related properties for the other, or estimate certain parameters for one walk via related parameters of the other (e.g.~\cite{alon,ram,cogrowth2}).  Conversely, when $G$ is non-regular the two walks can exhibit quite different behaviors (e.g.~\cite{beres,GW}).

\medskip  

Indeed, when $G$ is $d$-regular for $d \ge 3$, a well-known  argument involving the universal cover  of $G$ shows that the SRW $(S_n)_{n=0}^{\infty}$ and the NBRW  $(S_n^{\mathrm{BN}})_{n=0}^{\infty}$ on $G$ can be coupled so that $S_n^{\mathrm{BN}}=S_{\tau_n}$ for all $n$, where $(\tau_{n+1}-\tau_n)_{n=1}^{\infty} $ are i.i.d.\ random variables satisfying that $\Pr(\tau_2-\tau_1 > s) \le e^{-cs}$ for all $s \ge 0$  for some $c>0$ (see Proposition \ref{prop:reg}). However, there does not exist such a coupling when $G$ is non-regular or when the NBRW is replaced by a $k$-NBRW, provided that $k > \mathrm{girth}(G)$ (where the \emph{girth} of a graph $G$, denoted by $\mathrm{girth}(G)$, is defined as the length of the shortest cycle in $G$).

One of the most fundamental questions that one can ask about any random walk is whether it is recurrent or transient. 
 A fairly simple consequence of the aforementioned coupling is that when $G$ is regular SRW is recurrent iff the NBRW is recurrent. 
\begin{proposition}
\label{prop:reg}
Let $G=(V,E)$ be a $d$-regular connected graph for some $d \ge 3$. Then the SRW on $G$ is recurrent iff the NBRW on $G$ is recurrent.
\end{proposition}  
Proposition \ref{prop:reg} is proved in \S\ref{s:coupling}. It is natural to ask whether the same assertion holds even without regularity or when the NBRW is replaced with a $k$-NBRW.   
 In this work we show that under some mild conditions the answer to the last question is positive. To the best of the author's knowledge, this was unknown even for $k=1$. Moreover, we show (Theorem \ref{thm:Cayley}) that for Cayley graphs of finitely generated infinite Abelian groups which are not isomorphic to $\Z$ we have for all $k$ that the SRW is recurrent iff the $k$-NBRW is recurrent . Finally, we show (Theorem \ref{thm:pNBRW}) that for any connected graph of bounded degree   the variant of the NBRW which backtracks at each step with probability $p \in (0,1) $ is recurrent iff the SRW is recurrent.

 A major difficulty in studying the $k$-NBRW is that even for $k=1$ it  is non-reversible (see \S\ref{s:Markov} for the definition of reversibility).  It is much harder to apply techniques from potential theory to non-reversible Markov chains. One reason is that in order to apply techniques from potential theory one first has to know a stationary  measure for the chain. For non-reversible Markov chains it is sometimes hard to find a stationary measure, and if the chain is transient, it is possible that no stationary measure exists. In particular, for $k>1$ it seems quite difficult to find a stationary measure for the vertex $k$-NBRW. In fact, this is also true for the edge $k$-NBRW if it may get stuck. 

Another reason is the lack of rich variety of comparison  techniques available in the reversible setup (see \S\ref{s:comparison}). However,  one available comparison  technique which we shall exploit is that the additive symmetrization of a recurrent Markov chain is always recurrent (see Theorem \ref{thm:AG}). Unfortunately, there exist only few techniques for establishing the converse implication,  or more generally for arguing that the geometry of a non-reversible chain is in some sense equivalent to that of a reversible one. 

To overcome this difficulty, 
  we show that if $G$ has a positive density of short simple cycles then viewed at ``nice" random times, the edge NBRW  becomes a reversible chain which can be compared with SRW on $G$. So is the case for the $k$-NBRW under slightly more complicated conditions. 

\subsection{Basic definitions - SRW, NBRW and recurrence/transience}
\label{s:defofwalks}
Let $G=(V,E)$ be a locally finite connected (simple)\footnote{Although all of our results can be extended to the case that $G$ is a multigraph, for the sake of simplicity of notation and clarity of presentation we shall assume that $G$ contains no multiple edges nor loops.} graph.
A \emph{simple random walk} (\textbf{SRW}) on $G$ is a reversible Markov chain  with state space $V$ and transition probabilities: $P(x,y)=\Ind{x \sim y}/\deg(x)$, where $x \sim y$ indicates that $\{x,y\} \in E$, and $\deg (x)$ is the degree of $x$ (which is the number of edges which are incident to $x$).    A \emph{non-backtracking random walk} (\textbf{NBRW})  $(X_t)_{t=0}^{\infty}$ on $G$ started from vertex $u$ makes its first step according to the same rule as SRW started from $u$, and then evolves according to the following rule: the conditional distribution of  $X_{t+2}$ given $(X_i)_{i=0}^{t+1}$ is chosen from the uniform distribution on $\{z \in V :z\sim X_{t+1},z \neq X_t \}$, unless this set is empty (i.e.~$\deg(X_{t+1}=1)$), in which case $X_{t+2}=X_t$ (i.e., the walk is forced to backtrack). While this is \textbf{not} a Markov chain, it can be transformed into a Markov chain on  $\overrightarrow{E} $, the set of directed edges of $G$, whose transition probabilities are given by
\begin{equation}
\label{e:B}
B((x,y),(z,w)):=\frac{\Ind{y=z,w \neq x } }{\deg(y)-1}+\Ind{y=z,w = x,\deg(y)=1}.
\end{equation} 
Let $k \in \N$. Similarly, an \emph{edge} $k$\emph{-NBRW} (respectively, a \emph{vertex} $k$\emph{-NBRW})  $(X_t)_{t=0}^{\infty}$ evolves according to the following rule: the conditional distribution of  $X_{t+1}$ given $(X_i)_{i=0}^{t}$ is chosen from the uniform distribution on  \[\{z \in V :z\sim X_{t},\, \{z , X_t\} \notin \{ \{X_{i-1},X_i \}: t-k<i \le t  \}  \}\]
\[(\text{respectively,} \quad \{z \in V :z\sim X_{t}, \, z \notin \{ X_i : t-k \le i \le t  \}  \}), \]
unless this set is empty, in which case we say the walk is \emph{stuck}. If the walk is stuck then $X_{t+1}$ is chosen from the uniform distribution over the neighbors of $X_t$ (this is done in order to ensure the $k$-NBRW does not have absorbing states). These walks can be transformed into Markov chains on a subset of the set of directed paths of length $k$ in an analogous manner to the case $k=1$. We denote the collection of all paths of length $ k$ in the graph  $G=(V,E)$ (where $G$ will be clear from context) by
\[\mathcal{P}_{k}:=\{(x_i)_{i=0}^{k} \in V^{k+1} : \{x_{i} , x_{i-1}\} \in E \text{ for all }i \in [  k] \}. \] 
Let
\begin{equation}
\label{e:1.2}
\begin{split}
&\Pke:=\{ \gamma \in \PP_k: \gamma \text{ crosses $k$ distinct edges} \} \quad \text{and}
\\ &  \Pkv:=\{ \gamma \in \PP_k: \gamma \text{ contains $k+1$ distinct vertices} \}.
 \end{split}
 \end{equation}
We may take the state space of the edge (respectively, vertex) $k$-NBRW to be the collection of all length $k$ paths which are accessible from $\Pke$ (respectively, $\Pkv $). If the walk cannot get stuck then the state spaces are simply  $\Pke$ and $\Pkv $, respectively (see \S\ref{s:MC} for details).

\begin{remark}
Note that (when $G$ is simple) the NBRW is both a vertex and an edge 1-NBRW. However for $k>1$  the restrictions imposed on the vertex $k$-NBRW are stronger than those imposed on the edge $k$-NBRW. Moreover, the former is less symmetric (the latter satisfies a symmetry relation \eqref{e:pike2} similar to the detailed balance equation, which the former does not) and thus harder to analyze. 
\end{remark}
 
\begin{remark}
\label{r:nonstuck}
In what comes we shall often assume that the $k$-NBRW cannot get stuck. Note that the vertex (respectively, edge) $k$-NBRW cannot get stuck if the minimal degree of the graph is at least $k+1$ (respectively, $\lceil 2k/3 \rceil +1 $). 

Another  condition which ensures that the edge $k$-NBRW cannot get stuck is that all vertices are either of even degree or of degree at least $\lceil 2k/3 \rceil +1$. To be precise, (under the last condition) the only  vertex at which the edge $k$-NBRW might potentially get stuck is the  starting point, but this is not the case in the Markov chain representation of the edge $k$-NBRW, provided that the starting state is a path that does not get stuck. Consequently, the collection of graphs for which the edge $k$-NBRW cannot get stuck is large.    
\end{remark} 

We say that a vertex $v$ is \emph{recurrent} for a certain type of random walk on a graph $G$ (e.g.\ SRW and edge/vertex $k$-NBRW) if started from $v$ the walk returns to $v$ infinitely often $\mathrm{a.s.}$. Otherwise, vertex $v$ is said to be \emph{transient} for that walk type. We say that $G$ is recurrent (respectively, transient) for a certain type of random walk if all of its vertices are recurrent (respectively, transient) for that walk.
\begin{remark}
When the $k$-NBRW is viewed as a Markov chain, its state space is not the vertex set, but a certain subset of the collection of all paths of length $k$. In all of the cases considered in our main results, Theorems \ref{thm:main}-\ref{thm:pNBRW}, the above notion of recurrence for the $k$-NBRW is in fact equivalent to the usual notion of recurrence for a Markov chain. Thus there is no ambiguity in the definition of recurrence/transience.
\end{remark} 

\subsection{Our results}
\label{s:results}
Before stating our results we first list  some conditions for later reference.     
\begin{itemize}
\item[$(1)$]  There exists some $L>0$ such that every $2$\emph{-path} is of length at most $L$, where a $2$-path is a path consisting of degree 2 vertices.
\item[$(2)$] There exists some $R>0$ such that every ball of radius $R$ contains a cycle.
\item[$(2_k^{\mathrm{e}})$]  There exists some $R>0$ such that for every $\gamma,\gamma' \in \Pke$ which contain at least one common vertex, the edge $k$-NBRW can reach $\gamma'$ from $\gamma$ in at most $R$ steps. 
\item[$(2_{k}^{\mathrm{v}})$] There exists some $R>0$ such that for every $\gamma,\gamma' \in \Pkv$ which contain at least one common vertex, the vertex $k$-NBRW can reach $\gamma'$ from $\gamma$ in at most $R$ steps.\item[$\ekk $] The  edge $k$-NBRW on $G$ cannot get stuck.
\item[$\vkk $] The  vertex $k$-NBRW on $G$ cannot get stuck.
\end{itemize}
Observe that conditions $(2_k^{\mathrm{e}})$ and $(2_k^{\mathrm{v}})$ both imply condition $(1)$. Condition $(1)$ will not be assumed in any of our main results. 

 Let $v_0,v_1,\ldots,v_{\ell}=v_0 \in V $ be such that $\{v_{i-1},v_{i}\} \in E $ for all $i \in [\ell]:=\{1,2,\ldots,\ell\} $.  We say that $(v_0,v_1,\ldots,v_{\ell}=v_0)$ is an \emph{edge} (respectively, \emph{vertex}) \emph{simple cycle} (of length $\ell$) if $\{v_{i-1},v_{i}\} \neq \{v_{j-1},v_{j}\}$ for all $i \neq j \in [\ell ]$ (respectively, if $|\{v_{i} : i \in [\ell ] \}|=\ell$). 
     
\begin{itemize}
\item[$(4_k^{\mathrm{e}})$]  There exists some $R>0$ such that every ball of radius $R$ contains an edge simple cycle of length at least $k$.
\item[$(4_{k}^{\mathrm{v}})$] There exists some $R>0$ such that every ball of radius $R$ contains a vertex simple cycle of length at least $k+1$.
\end{itemize}
We believe that conditions $(2_k^{\mathrm{e}})$ and $(2_k^{\mathrm{v}})$ are in fact equivalent to conditions  $(4_k^{\mathrm{e}})$ and $(4_k^{\mathrm{v}})$, respectively. Conditions  $(4_k^{\mathrm{e}})$ and $(4_k^{\mathrm{v}})$ will not be used in the paper.

The next theorem concerns with the case that   $G$ is of (uniformly) \emph{bounded degree} (i.e., $\sup_{v \in V }\deg (v) < \infty $).   
\begin{maintheorem}
\label{thm:main}
Let $G$ be a connected  graph of bounded degree.
\begin{itemize}
\item[(i)]
If the SRW on $G$ is transient then the NBRW on $G$ is also transient. 
\item[(ii)] If condition $(2)$ holds then the SRW on $G$ is transient iff the NBRW on $G$ is transient.
\item[(iii)]
If conditions  $\ek$ and $\ekk$  hold, then the SRW on $G$ is transient iff the edge $k$-NBRW is transient.
\item[(iv)]
If  $G$ is vertex-transitive, conditions  $\vk$ and $\vkk$ hold and  the SRW on $G$ is transient, then the vertex $k$-NBRW is also transient.

\end{itemize}
\end{maintheorem}

Let $H$ be a countable group. Let $S$ be a finite symmetric (i.e., $S=S^{-1}:=\{s^{-1}:s \in S \}$) set of generators of $H$ (that is, $H=\{s_1 \cdots s_n:n \ge 0,s_1,\ldots s_n \in S \}$). The (right)
 \textbf{Cayley graph} of $H$ w.r.t.~$S$ is $H(S):=(H,E(S))$, where $E(S):= \{\{h,hs \}:h \in H,s \in S \}$.  As the following theorem asserts, when  $G$ is a Cayley graph of an Abelian group, no additional assumptions are necessary for the assertion of Theorem \ref{thm:main} to hold.

\begin{maintheorem}
\label{thm:Cayley}
Let $H$ be a finitely generated infinite Abelian group which is not isomorphic to $\Z$.  Let $G=H(S)$ be the Cayley graph of  $H  $ w.r.t.~some finite symmetric set of generators $S$.  Let $k \in \N$. Then the following are equivalent: 
\begin{itemize}
\item[(i)]  The SRW on $G$ is recurrent.
\item[(ii)] The edge $k$-NBRW on $G$ is recurrent.   
\item[(iii)] The vertex $k$-NBRW on $G$ is recurrent. 
\end{itemize} 
\end{maintheorem}
\begin{remark}
While the assertion of Theorem \ref{thm:Cayley} clearly fails when $H = \Z $ and $S=\{ \pm 1 \} $, it is not hard to modify the proof of Theorem \ref{thm:Cayley}  to show that if  $H = \Z $ and $|S|>2 $  then for all $k \in \N$ the edge $k$-NBRW is recurrent. 
\end{remark}
\begin{remark}
It is classical that whether SRW on $H(S)$ is recurrent or not does not depend on $S$ (that is, if $S_1,S_2$ are two finite symmetric sets of generators of $H$ then the SRW on $H(S_1)$ is recurrent iff the SRW on $H(S_2)$ is recurrent). It follows from Theorem \ref{thm:Cayley} that if $H \neq \Z $ is Abelian, the same is true for the edge and vertex  $k$-NBRWs for all $k \in \N$. 
\end{remark}
We now define a variant of the  NBRW for which condition $(2)$ is not required for the assertion of Part (ii) of Theorem \ref{thm:main} to hold.  
\begin{definition}
\label{d:BRW}
Let $G=(V,E)$ be a locally-finite connected graph. Fix some $p \in (0,1)$. Consider a walk on the set of directed edges $\dE $ whose transition kernel is given by
\[B_{p}((x,y),(z,w) ):=(1-p)B\left((x,y),(z,w) \right)+p \cdot \Ind{z=y,w=x } , \]
where $B$ is the transition kernel of the NBRW on $G$ (i.e., at each step with probability $p$ the walk backtracks and otherwise it evolves  like a NBRW). We call this walk a $p$-\emph{BRW}.
\end{definition}
While the following theorem is of self-interest, its proof  will be used to highlight the key ideas behind  the proofs of Theorems  \ref{thm:main} and  \ref{thm:Cayley}  in a simpler setup. 
\begin{maintheorem}
\label{thm:pNBRW} 
Let $G=(V,E)$ be a connected graph of bounded degree. Then for every $p \in (0,1)$ the $p$-BRW on $G$ is transient iff the SRW on $G$  is transient.
\end{maintheorem}

\subsection{Open problems and remarks}
\label{s:OR}
Recall that two graphs $G_1,G_2$ are roughly-isometric if one can be embedded in the other in an ``almost surjective" manner while distorting distances by at most a constant additive and multiplicative factor (see Definition \ref{def: RI} for a precise formulation). Recall further that the relation of ``being roughly-isometric" is an equivalence relation which preserves  recurrence/transience (see Fact \ref{f:RIinvariance}). Moreover, it is not hard to verify that condition $(2)$ is invariant under rough-isometries.  Hence the following corollary is an immediate consequence of Theorems \ref{thm:main} and \ref{thm:pNBRW}. 
\begin{corollary} Let $G$ and $H$ be two connected roughly-isometric graphs of bounded degree. Then the following hold:
\begin{itemize}
\item[(i)]  If the SRW on $G$ is transient then the NBRW on $G$ and on $H$ are both transient.
\item[(ii)] If condition $(2)$ holds for $G$ then the NBRW on $G$ is transient iff the NBRW on $H$ is transient.
\item[(iii)] For every $p \in (0,1)$ the $p$-BRW on $G$ is transient iff the $p$-BRW on $H$ is transient.
\item[(iv)] If $G$ and $H$ both satisfy conditions  $\ek$ and $\ekk$, then the edge $k$-NBRW on $G$ is transient iff the edge $k$-NBRW on $H$ is transient. 
\end{itemize}  
\end{corollary}
We leave the following as open problems.
\begin{open}
Let $G$ be a connected  graph of bounded degree which satisfies conditions  $\vk$ and $\vkk$. Is it the case that the SRW on $G$ is transient iff the vertex $k$-NBRW is transient?
\end{open}
\begin{open}
Does the assertion of Theorem \ref{thm:Cayley} remain valid when $G$ is only assumed to be vertex-transitive?
\end{open}
\begin{open}
Let $G$ be a connected  graph of bounded degree satisfying condition $(1)$. Is it the case that the SRW on $G$ is transient iff the NBRW on $G$ is transient? \end{open}
The following example demonstrates that if condition $(1)$ fails then it is possible that the SRW is recurrent while the NBRW is transient.\begin{example}
\label{ex:long2paths}
Let $H=(V,E)$ be some transient graph of bounded degree. Fix some $o \in V$. Let $\Pi_n$ be the collection of all edges with one end-point of distance $n-1$ from $o$ and the other of distance $n$.   Let $G$ be the graph obtained from $H$ by replacing for all $n \in \N $ each edge in $\Pi_n$ by a path of length $|\Pi_n| $. 

Observe that when the NBRW on $G$ enters one of its 2-paths  it must reach the other end of that 2-path. Thus if we observe the NBRW on $G$ only when it visits sites belonging to $V$ we obtain a realization of the NBRW on $H$. Since the SRW  on $H$ is transient, by Theorem \ref{thm:main} so is the NBRW on $H$ and  hence the NBRW on $G$ must  be transient as well. 

To see that the SRW on $G$ is recurrent, first apply a standard network reduction,  replacing each 2-path added to $G$ by a single edge of edge-weight  which equals the inverse of the length of that path (obtaining a network on $H$ in which each edge in $\Pi_n$  has edge-weight $1/|\Pi_n|$). Then apply the Nash-Williams criterion (e.g.~\cite[(2.13)]{lyons}) to the cut sets $\Pi_1,\Pi_2,\ldots $.  
\end{example}
To end the introduction, the author would like to thank Itai Benjamini for suggesting
most of the problems considered in this paper (private communication). 

\subsection{Organization of the paper}
In \S\ref{s:Markov} we present some background on  Markov chains.  In \S\ref{s:2} we introduce additional notation and describe a representation of the $k$-NBRW as a Markov chain.  In \S\ref{s:comparison} we review a variety of comparison techniques. In section \ref{s:SOLG} we apply some of these techniques in order to compare SRW with the additive symmetrization of the $k$-NBRW. In \S\ref{s:overview} we give an overview of the main ideas behind the proofs of our main results and prove Theorem \ref{thm:pNBRW}. In \S\ref{s:findingpi} we find a stationary measure for the edge $k$-NBRW (under condition $\ekk$). In \S\ref{s:proofthm1}-\ref{s:proofthm2} we prove Theorems \ref{thm:main}-\ref{thm:Cayley}, respectively.

\section{Some definitions related to Markov chains}
\label{s:Markov}

We now recall some basic definitions related to Markov chains. For further background on Markov chains see e.g.~\cite{aldous,levin,lyons}. We call  $(V, E,(c_{e})_{e
\in E})$ a \emph{network}  if $(V,E)$ is a graph (we allow it to contain loops, i.e.~$E \subseteq \{\{ u,v\}:u,v \in V \} $) 
and $(c_{e})_{e
\in E}$ are symmetric edge weights (i.e., $c_{u, v}=c_{v,u}>0$ for every $\{u,v \} \in E$ and $c_{u,v}=0$ if $\{u,v\} \notin E$) such that  $c_v:=\sum_{w} c_{v, w}< \infty $ for all $v \in V $. We say the network is connected if the graph  $(V,E)$ is connected.  The  (weighted) nearest neighbor   random walk corresponding to    $(V, E,(c_{e})_{e
\in E})$ repeatedly does the following: when the current state is $v\in V$, the walk will move to vertex $u$  with probability $c_{u, v}/c_{v}$. 
The  choice  $c_{e}=\Ind{e \in E }$ corresponds to SRW  on $(V,E)$. We denote this network by $(V,E,(1)_{e \in E})$. 

\medskip

Consider an ergodic Markov chain  $(X_k)_{k=0}^{\infty}$  on a countable state space $V$ with a stationary  measure $\pi$ and transition probabilities given by $P$. We say that $P$ is \emph{reversible} w.r.t.\ $\pi$ if $\pi(x)P(x,y)=\pi(y)P(y,x)$ for all $x,y \in V$. Observe that a weighted nearest-neighbor random walk is reversible w.r.t.~the measure $\pi$ given by $\pi(v):=c_v$ and hence $\pi$ is stationary for the walk. Conversely, every reversible chain can be presented as a network with weights $c_{x,y}=\pi(x)P(x,y)$.
 The \emph{time-reversal} of $P$ (w.r.t.\ $\pi$), denoted by $P^*$, is given by $\pi(x)P^*(x,y):=\pi(y)P(y,x)$ and its \emph{additive symmetrization} (w.r.t.\ $\pi$) is a reversible Markov chain (w.r.t.\ $\pi$) whose transition probabilities are given by $S:=\frac{1}{2}(P+P^*)$.\footnote{Note that when the chain is recurrent, $\pi$ is unique up to a constant factor, and so $P_*$ and $S$ are uniquely defined. This may fail when the chain is transient! A transient Markov chain may fail to have a stationary measure and may have two different stationary measures which are not constant multiples of one another.  Consider a random walk on $\Z$ with $P(i,i+1)=p=1-P(i,i-1)$ for all $i \in \Z$, with $p \in ( 1/2,1) $. Note that $P$ is reversible w.r.t.\ $\pi(i)=(p/(1-p))^i $ and is stationary also w.r.t.\ the counting measure $\mu$ on $\Z$. The additive symmetrization w.r.t.\ $\pi$ is again $P$ while w.r.t.\ $\mu$ is SRW on $\Z$.  }

The \emph{hitting time} of a set $A \subset V$ is $T_A:=\inf \{t \ge 0 :X_t \in A \}$. Similarly, let $T_A^+:=\inf \{t \ge 1 : X_t \in A \} $. When $A=\{x\}$ is a singleton, we write $T_x$ and $T_x^+$ instead of $T_{\{x\}}$ and $T_{\{x\}}^+$. For a set $A \subset V $ such that $\Pr_v[T_A^+<\infty]=1 $ for all $v \in V$ (under irreducibility we may replace `for all' by `for some') the \emph{induced chain on} $A  $ is a Markov chain with state space $A$ and transition probabilities given by $Q_A(a,b):=\Pr_a[X_{T_{A}^+}=b]$, where  $\Pr_a$ denotes the law of the entire chain, started from vertex $u $. It is not hard to verify that if the original chain is reversible w.r.t.\ $\pi$, then  the induced chain on $A$ is reversible w.r.t.\ to the restriction of $\pi$ to $A$ (e.g.\ \cite[p. 186]{levin}). In fact, if the original chain is recurrent then this holds even without reversibility. Indeed in this case the stationary measure $\pi$ is unique up to a multiplication by a constant factor and is given by $\pi(y)= \E_{x}[|\{t \in [ 0,X_{T_x^+}) :X_t =y \} |] $ (where $x$ is arbitrary). Taking $x,y \in A$ we see that the r.h.s.\ is the same as the corresponding expectation for the induced chain on $A$, and so the stationary measure of the induced chain on $A$ is simply the restriction of $\pi$ to $A$ as claimed.

We now assume that $\sum_{v \in V} \pi(v)=\infty$. The $\ell_2$ and  $\ell_{\infty}$ norms of $f \in \R^V  $ are given by $\|f \|_2:=\sqrt{\langle f,f \rangle_{\pi}} $ and  $\|f \|_\infty:=\sup_{v \in V }|f(v)| $, where   $\langle f,g \rangle_{\pi}:= \sum_{v \in V}\pi(v) f(v) g(v)$. The space of $\ell_2$ functions is given by $\mathcal{H}:=L_2(V,\pi)=\{f :\R^V:\|f \|_2< \infty \}$.  Then $P$ (and similarly, also $S$ and  $P^*$)  defines a linear operator on $\mathcal{H} $ via $Pf(x):= \sum_{y \in V} P(x,y)f(y)=\E_x[f(X_1)]$. Note that $P^*$ is the dual of $P$ and that $S$ is self-adjoint. The \emph{Dirichlet form} $\mathcal{E}_P(\cdot,\cdot):\mathcal{H}^2 \to \mathbb{R}$ corresponding to $P$    is  $\mathcal{E}_{P}(f,g):=\langle (I-P)f,g \rangle_{\pi}$. Note that \begin{equation}
\label{e:Dirichlet}
\mathcal{E}_{P}(f,f)=\mathcal{E}_{P^*}(f,f)=\mathcal{E}_{S}(f,f)=\frac{1}{2} \sum_{x,y}\pi(x)S(x,y)(f(x)-f(y))^2. \end{equation}

\section{Notation and a construction of the $k$-NBRW as a Markov chain}
\label{s:2}
\subsection{Notation}
\label{s:notation}
Fix some graph $G=(V,E)$. Here we allow $G$ to contain loops (however we do not allow it to have multiple edges between a pair of vertices). Recall that we write $u \sim v$ iff $\{u,v\} \in E$. We denote a directed edge from $x$ to $y$ by $(x,y)$ and an undirected edge between $x$ and $y$ by $\{x,y\}$. We denote the set of directed edges by $\dE:=\{(x,y):\{x,y\} \in E \}$. For $\de=(x,y) \in \overrightarrow{E}  $ we denote its \emph{reversal}, tail and head (respectively) by  $\de^{\rr}:=(y,x)$, $\de_{-}:=x$ and $\de_+:=y$. 
Let $k \in \N$. We denote the set $\{1,2,\ldots,k \}$ by $[k]$.    Recall that we denote the collection of all paths of length $ k$ in a graph $G$  by
\[\mathcal{P}_{k}:=\{(x_i)_{i=0}^{k} \in V^{k+1} : \{x_{i} , x_{i-1}\} \in E \text{ for all }i \in [  k] \}. \]
Let $\gamma=(\gamma_0,\ldots,\gamma_k) \in \PP_k $. We denote the $i$th coordinate of $\gamma$ by  $\gamma_i$. The \emph{reversal} of $\gamma$ is given by $\gamma^{\mathrm{r}}:=(\gamma_k,\ldots,\gamma_0)$ (that is, $\gamma_i^{\mathrm{r}} :=\gamma_{k-i}$ for all $i$). Its length is denoted by $|\gamma |$ (where $|\gamma|+1 $ is the number of coordinates in $\gamma$). For $e \in E$ (respectively, $v \in V$) we write $e \in \gamma $ (respectively, $v \in \gamma$) if $e \in \gamma^{\mathrm{e}}:=\{\{\gamma_{i-1},\gamma_i \}:  i \in [k] \}$ (respectively, $v \in \gamma^{\mathrm{v}}:=\{\gamma_i: 0 \le i \le k \}$). For  $\gamma =(\gamma_0,\ldots,\gamma_k) \in \PP_k$ and  $v \sim \gamma_k $   we write $v \sime \gamma$ (respectively, $v \simv \gamma $) if $\{\gamma_k,v \} \notin \gamma $ (respectively, $v \notin \gamma $). We also write $\gamma(v):=(\gamma_1,\ldots,\gamma_{k},v) $ (note that $\gamma(v) \in \PP_k $ and $\gamma(v)_i:=\gamma_{i+1}$ for  $0 \le i <k$). As in \eqref{e:1.2} let
\[\Pke:=\{ \gamma \in \PP_k: | \gamma^{\mathrm{e}} |=k \} \quad \text{and} \quad \Pkv:=\{ \gamma \in \PP_k: | \gamma^{\mathrm{v}} |=k+1 \}. \]
Observe that  $\gamma \in \Pke $ (respectively, $\gamma \in \Pkv$) iff $\gamma^{\rr} \in \Pke $ (respectively, $\gamma^{\rr} \in \Pkv$). 
For $\gamma \in \PP_k$ let $A(\gamma):=\{\gamma(v):v \sim \gamma_k \}$,
\[ N^{\mathrm{e}}(\gamma):=\{\gamma(v): v \sime \gamma \}  \quad \text{and} \quad N^{\mathrm{v}}(\gamma):=\{\gamma(v): v \simv \gamma \}. \]
We denote the transition kernel of the edge (respectively, vertex) $k$-NBRW by $P_{k,\ee} $ (respectively, $P_{k,\vv}$).
For every set $A$ we denote the counting measure on $A$ (i.e., the measure which assigns to every $a \in A $ mass 1) by $\pi_{A}$.
\subsection{A representation of the $k$-NBRW as a Markov chain}
\label{s:MC}
Let $k \ge 1$. In this section we construct the edge and vertex $k$-NBRWs as Markov chains. 
 We first construct the edge and vertex $k$-NBRWs  as Markov chains on $\PP_k$ and then later we restrict the state space to include only the ``relevant" paths.
Given that the current location of the walk is $\gamma$ the edge  (respectively, vertex) $k$-NBRW evolves according to the following rule: If $N^{\mathrm{e}}(\gamma) \neq \eset $ (respectively, $N^{\mathrm{v}}(\gamma) \neq \eset $),  the walk moves to some $\gamma' \in N^{\mathrm{e}}(\gamma) $ (respectively, $\gamma' \in N^{\mathrm{v}}(\gamma)  $) chosen from the uniform distribution. Otherwise, it moves to some $\gamma' \in A(\gamma) $ chosen from the uniform distribution. 

We take the state space of the edge (respectively, vertex) $k$-NBRW to be  \[\Omega_k^{\ee}:=\{\gamma' \in \PP_k:\exists \, \gamma \in \Pke \text{ and }n \ge 0 \text{ such that }P_{k,\ee}^n(\gamma,\gamma')>0 \}\] (respectively, $\Omega_k^{\vv}:=\{\gamma' \in \PP_k:\exists \, \gamma \in \Pkv \text{ and }n \ge 0 \text{ such that }P_{k,\vv}^n(\gamma,\gamma')>0 \}$),  

the collection of all states which are accessible starting from some state in $\Pke $ (respectively, $\Pkv$).

\section{Comparison Techniques}
\label{s:comparison}
In this section we present various techniques which allow one to deduce that a certain Markov chain is recurrent (or transient) if some other chain is recurrent (or transient).

The \emph{capacity} of a finite  set $A$ (w.r.t.\ a certain Markov chain and some stationary measure  $\pi$) is defined as \[\mathrm{Cap}(A):= \sum_{a\in A}\pi(a) \Pr_a[T_{A}^{+}=\infty] .\]
\subsection{Comparison of two reversible Markov chains}
\label{s:comrev}
Under reversibility, by the Dirichlet principle we have that for a finite set $A$
\begin{equation}
\label{e:DP}
\mathrm{Cap}(A)=\inf_{f \in \R^V :\, f \upharpoonright A \equiv 1, \, 0 \le f \le 1, \, \mathrm{supp}(f) \text{ is finite}} \mathcal{E}_{P}(f,f),
\end{equation}
where $f \upharpoonright A$ is the restriction of $f$ to $A$, $\mathrm{supp}(f):=\{v \in V:f(v) \neq 0 \}$ is the support of $f$ and $\mathcal{E}_{P}(f,f)$ is as in \eqref{e:Dirichlet}. The following standard lemma is an immediate consequence of \eqref{e:DP}.
\begin{lemma}
\label{lem:comparison1}
Let $P,Q$ be two transition kernels on the same (countable) state  space $\Omega$. Assume that $P$ and $Q$ are reversible w.r.t.~$\pi_P$ and $\pi_Q$, respectively. Assume further that for all $f : \Omega \to \R $ we have that $\mathcal{E}_{P}(f,f) \le M  \mathcal{E}_{Q}(f,f) $. Then for all finite $A \subset \Omega$ we have that \[\mathrm{Cap}_P(A) \le M \mathrm{Cap}_{Q}(A), \]
where $\mathrm{Cap}_P(A)  $ and $\mathrm{Cap}_Q(A)$ are the capacities of $A$ w.r.t.~$P$ and $Q$, respectively. In particular, if $Q$ is recurrent then so is $P$. 
\end{lemma}
\begin{definition}
\label{d:flow}
Let $G:=(V,E,(c_e)_{e \in E})$ and  $G':=(V,E',(c'_e)_{e \in E'})$  be two connected networks.   For every $(x,y) \in \dE' $ let $\PP_{x,y}(G)$ be the collection of all non-empty (oriented) paths from $x$ to $y$ in $G$.  Denote the number of times the (directed) edge $(u,v)$ appears in a path $\gamma$  by \[r((u,v),\gamma):=|\{i: \gamma_i=u,\gamma_{i+1}=v  \}| .\]  

We say that $\Phi:=(\Phi_{x,y})_{(x,y) \in \dE' }$ is a $G$-$G'$ \emph{flow} (or $P$-$P'$ \emph{flow}, where $P$ and $P'$ are the transition kernels of $G$ and $G'$, respectively) if for every $\{x,y\} \in E'$ the following hold:
\begin{itemize}
\item[(i)]  $\Phi_{x,y}$ is a map from $\PP_{x,y}(G)$ to $[0,c_{x,y}']$, \item[(ii)] $\sum_{\gamma}\Phi_{x,y}(\gamma)=c_{x,y}' $ and 
\item[(iii)]  $\Phi_{x,y}(\gamma)=\Phi_{y,x}(\gamma^{\rr}) $ for every $\gamma$.  \end{itemize}
The \emph{congestion} of $(u,v) \in \dE $ w.r.t.~$\Phi $ is defined as
\[A_{u,v}(\Phi):=\frac{1}{c_{u,v}}\sum_{(x,y) \in \dE' }\sum_{\gamma \in \PP_{x,y}(G)}r((u,v),\gamma)|\gamma|\Phi_{x,y}(\gamma). \]
The \emph{congestion} of $\Phi$ is defined as $A(\Phi):=\sup \{A_{u,v}(\Phi):(u,v) \in \dE \} $.
\end{definition}
The following lemma provides a standard technique for utilizing Lemma \ref{lem:comparison1}. The lemma is essentially Theorem 2.3 from \cite{comparison}. While the formulation in \cite{comparison} is slightly different, the difference is non-essential (cf.~Theorem 1  in \cite{comparison2} and the remark thereafter).  
\begin{lemma}[]
\label{lem:comparison2}
In the setup of Definition \ref{d:flow} Let $P$ and $P'$ be two reversible transition kernels on the same state space. Let $\Phi$ be a $P$-$P'$ \emph{flow}. Then $\mathcal{E}_{P'}(f) \le A(\Phi) \mathcal{E}_{P}(f) $ for all $f \in \R^V$.
\end{lemma}
We now present a variant of Lemma \ref{lem:comparison2}, useful for comparing two networks with different state spaces.
\begin{definition}
\label{d:RILP}
Let $G:=(V,E,(c_e)_{e \in E})$ and  $G':=(V',E',(c'_e)_{e \in E'})$  be two connected networks. Using notation from Definition \ref{d:flow}, we say that $\phi:V' \to V $ is a \emph{rough-embedding} from $G'$ to $G$ if there exist a map $\Phi: \dE' \to \cup_{x,y:\,(x,y)\in \dE'} \PP_{\phi(x),\phi(y)}(G) $ and some constants $\alpha,\beta>0$ such that the following hold:   
\begin{itemize}
\item[(i)] $\Phi((x,y))\in \PP_{\phi(x),\phi(y)}(G) $ and   $\Phi((y,x)) $ is its reversal for every $(x,y) \in \dE'$, 
\item[(ii)]  for every $(x,y) \in \dE' $ we have that $\frac{1}{c_{x,y}} \le \alpha \sum_{e \in \Phi((x,y))  }\frac{1}{c_e} $, and 
\item[(iii)]  every $(u,v) \in \dE $ belongs to the image (under $\Phi$) of at most $\beta $ edges $(x,y) \in \dE'$.
\end{itemize} 
\end{definition}
\begin{fact}[e.g.~\cite{lyons} Theorem 2.17]
\label{f:RILP}
If there is a rough-embedding from $G'$ to $G$ and $G'$ is transient, then $G$ is also transient.
\end{fact}
We now recall the notion of a rough-isometry. 

\begin{definition}
\label{def: RI}
Let $G_i:=(V_i,E_i)$ (where $i=1,2$) be two bounded degree graphs. For $u,v \in V_i$, let $d_i(u,v)$ be graph distance (w.r.t.~$G_i$) between $u$ and $v$ (i.e., the number of edges along the shortest path in $G_i$ between $u$ and $v$). We say that $f:V_1 \to V_2$ is a $K$-\emph{rough-embedding} from $G_1$ to $G_2$ if the map $f$ can stretch or shrink distances only by a multiplicative and additive factor of at most $K$ in the following sense: 
\begin{equation}
\label{e:d1d2}
 \forall \, u,v \in V_1, \quad  \sfrac{1}{K} d_1(u,v)-1 \leq d_2(f(u),f(v)) \leq K(d_1(u,v)+1). \end{equation} We say that it is a  $K$-\emph{rough-isometry} from $G_1$ to $G_2$ if \eqref{e:d1d2} holds and in addition $f$ is also ``almost surjective" in the following sense: For every $w \in V_2$, there exists some $v \in V_1$ such that $d_2(f(v),w) \leq K $. We say that $f$ is a \emph{rough-embedding} (respectively, \emph{rough-isometry}) if for some $K \in \N$ it is a $K$-rough-embedding (respectively, $K$-rough-isometry).   We say that $G_1 $ and $G_2$ are \emph{roughly-isometric} if there exists a rough-isometry between them.
\end{definition}
\begin{remark}
\label{r:RILP}
Let $G_i:=(V_i,E_i)$ (where $i=1,2$) be two bounded degree graphs. It is easy to see that if $f:V_1 \to V_2$ is a rough-embedding from $G_1$ to $G_2$ (in the sense of Definition \ref{def: RI}), then $f$ is also a rough-embedding from $(V_1,E_1,(1)_{e \in E_1 })$  to  $(V_2,E_2,(1)_{e \in E_2 })$  (i.e., from SRW on $G_1$ to SRW on $G_2$) in the sense of Definition \ref{d:RILP}.
\end{remark}
\begin{remark}
\label{r:ER}
Let $G_i:=(V_i,E_i)$ ($i=1,2$) be two bounded degree graphs.  It is easy to verify that  if  $f$  is a rough-isometry from $G_1$ to $G_2$, then any $g:V_2 \to V_1 $ such that $g(u) \in \{v \in V_1 : d_2(f(v),u)=\min_{v' \in V_1 } d_2(f(v'),u)\} $ for all $u \in V_2$, is also a rough-isometry. 
\end{remark}
In the bounded degree setup the relation of being roughly-isometric is an equivalence relation on graphs which preserves recurrence/transience (this can be derived from Fact \ref{f:RILP} and Remarks \ref{r:RILP}-\ref{r:ER}).

\begin{fact}[e.g.~\cite{lyons} Theorem 2.17 and Proposition 2.18]
\label{f:RIinvariance}
Let $G_i:=(V_i,E_i)$ (where $i=1,2$) be two connected graphs of uniformly bounded degree. Let $c^{(i)}:=(c_{e}^{(i)})_{e \in E}$ ($i=1,2$) be some edge weights which are uniformly bounded from above and below by some positive constants. 
\begin{itemize}
\item[(a)] The random walk corresponding to $(V_1,E_1,c^{(1)})$  is transient iff SRW on $G_1$ is.
\item[(b)] If there exists a rough-embedding from $G_1$ to $G_2$ and SRW on $G_1$ is transient, then also SRW on $G_2$ is transient.
\item[(c)] If there exists a rough-isometry from $G_1$ to $G_2$ then SRW on $G_1$ is transient iff  SRW on $G_2$ is transient.
\end{itemize}  
\end{fact}
We now recall the notion of `lumping' (or `short-circuiting') a network and the fact that this operation preserves transience.
\begin{definition}
\label{def:lumping}
Let  $G=(V, E,(c_{e})_{e
\in E})$ be an infinite connected network on $(V,E)$ with edge weights $(c_{e})_{e
\in E}$. Recall that for $v\in V$ we write $c_v:=\sum_{u:u \sim v }c_{u,v}$. Let $A_1,A_2,\ldots $ be a partition of $V$ into non-empty disjoint sets such that $ \sum_{v \in A_i}c_v< \infty$ for all $i$. Consider   the network  $G_{(A_{i})_{i \in \N}}:=(\N,\tilde E,(\tilde c_{e})_{e \in \tilde E })$ obtained from $G$ by collapsing each $A_i$ into a single state $i$ and setting $\tilde c_{i,j}:=\sum_{a_{i} \in A_i,a_{j} \in A_j }c_{a_i,a_j}$ for all $(i,j) \in \N^{2}$ (this operation is also called lumping $A_i$ together, or short-circuiting $A_i$, for all $i$) and setting  $\tilde E:=\{\{i,j\}:\tilde c_{i,j}>0 \} $. We say that $G_{(A_{i})_{i \in \N}}$ is a \emph{factor} of $G$.
\end{definition}
The following is a consequence of Rayleigh''s monotonicity principle.
\begin{fact}
\label{f:induced}
In the setup of Definition \ref{def:lumping}, if $G$ is transient then so is any factor of $G$.
\end{fact} 
\subsection{Comparison of a Markov chain to its additive symmetrization}
\label{s:s}
The connection between reversible networks,  electrical networks and potential theory is classical (see e.g., \cite{doyle} and \cite[Ch.~2]{lyons}). It was only in recent years that this connection was extended to the non-reversible setup in several extremely elegant works. The first progress on the non-reversible front was made by Doyle and Steiner \cite{doyle2} who derived an extremal characterization of the commute-time between two states, which shows that commute-times in the additive symmetrization of an irreducible chain cannot be smaller than the corresponding commute-times in the original chain. 

 Gaudilli\`ere and Landim \cite{nonrev} extended much of the classic potential theory to the non-reversible setup and derived several extremal characterizations for the capacity between two disjoint sets. In particular, they showed  \cite[Lemma 2.5]{nonrev} that the capacity between two disjoint sets of an irreducible Markov chain with a stationary measure $\pi$ is at least as large as that of the additive symmetrization of the chain.   
\begin{theorem}[\cite{nonrev} Lemma 5.1]
\label{thm:AG}
Consider an irreducible Markov chain  $\mathbf{X}=(X_k)_{k=0}^{\infty}$  on a countable state space $V$ with a stationary  measure $\pi$ and transition probabilities given by $P$. Let $\mathbf{X}^{s}=(X_k^{s})_{k=0}^{\infty} $ be the Markov chain corresponding to $S:=\frac{1}{2}(P+P^*)$, where $P^*(x,y):=\sfrac{\pi(y)}{\pi(x)}P(y,x)$. 
Then, if  $\mathbf{X}^{s}$ is transient then $\mathbf{X}$ is also transient.
\end{theorem}
We shall make use of Theorem \ref{thm:AG} only in the proof of Parts (i) and (iv) of Theorem \ref{thm:main}. For the rest of our results, our method of proof allows for a direct comparison of the considered non-reversible chain with some reversible Markov chain which can be compared with the SRW on the considered graph. 
\section{Equivalence between SRW on $G$ and the additive symmetrization of the $k$-NBRW}
\label{s:SOLG}
Let $G=(V,E)$ be a connected graph of bounded degree. In order to take advantage of Theorem \ref{thm:AG} we first need to compare the additive symmetrization of the $k$-NBRW with SRW on $G$.  In this section we do so by utilizing the machinery from \S\ref{s:comrev}-\ref{s:s}. We will later show that in all of the cases considered in Theorems \ref{thm:main} and \ref{thm:Cayley}  the $k$-NBRW has a stationary measure which is pointwise uniformly bounded from above and from below. Hence by Fact \ref{f:RIinvariance}, instead of the additive symmetrization of the $k$-NBRW, we may consider SRW on the graph supporting its transitions. We start with the case $k=1$.

 Recall that the \emph{line graph} of $G$ is defined as $G_{E}:=(E,E') $, where $\{e,e' \} \in E' $ iff $e \cap e' \neq \eset$ and $e \neq e'$. Recall that $\overrightarrow{E}:=\{(x,y):\{x , y\} \in E \}$ is the collection of directed edges of $G$. Following \cite{ortner} we define the \emph{symmetric oriented line  graph}  of $G$ to be $G_{\overrightarrow{E}}:= (\overrightarrow{E},F)$ the graph supporting the transitions of the NBRW on $G$. That is, $\{\de,\overrightarrow{f} \} \in F $ iff either
\begin{itemize}
\item[(a)] $\de_+=\df_-$ and $\Ind{\deg(\de_+)=1 \text{ or } \de_- \neq \df_+ }=1 $, or
\item[(b)]
  $\de_-=\df_+$ and $\Ind{\deg(\df_+)=1 \text{ or }\df_- \neq \de_+ }=1$. \end{itemize}
We shall make crucial use of the following simple observation: 
\begin{lemma}
\label{lem:simpleobservtion}
Let $G$ be a finite connected graph of bounded degree. The graph supporting the transitions of the additive symmetrization of the NBRW on $G$ is $G_{\dE} $. Moreover, both SRW on  $G_{\dE} $ and the additive symmetrization of the NBRW on $G$ are reversible w.r.t.\ the counting measure on $\dE $, and their minimal transition probability between neighboring directed edges is at least half the inverse of the maximal degree of $G$. In particular, either both are transient or both are recurrent.
\end{lemma}
\begin{proof}
The statement of the last sentence  follows from the rest of the lemma via Fact \ref{f:RIinvariance} (by noting that when the reversible measure is the counting measure, the edge-weights are simply the transition probabilities). The rest can be verified by inspection. 
\end{proof}
Before describing some relations between $G,G_E$ and $G_{\dE}$ we need a simple lemma. Its proof is deferred to \S\ref{s:sinkfree}. 
\begin{lemma}
\label{lem:sinkfree}
Let $G=(V,E)$ be an infinite locally finite connected  graph of minimal degree at least $2$. There exists a \emph{sink-free and source-free orientation} of $E$, i.e., a map $f:E \to \overrightarrow{E}$  such that: 
\begin{itemize}
\item[(i)] $f(\{u,v\})$ is either $(u,v)$ or $(v,u)$ for every $\{u,v\} \in E $, and 
\item[(ii)] every $v \in V$ is the tail of at least one $\overrightarrow{e} \in \dE $ in the image of $f$ and also the head of at least one  $\overrightarrow{e}$ in the image of $f$ (i.e.~$V=\{f(e)_+:e \in E \}=\{f(e)_-:e \in E \}$).
\end{itemize}
Recall that for $\gamma =(a,b)\in \dE$ its reversal $(b,a)$ is denoted by $\gamma^{\rr}$. For a graph $H=(V_{H},E_{H})$ and $x,y \in V_H$ we denote by $\mathrm{dist}_{H}(x,y)$ the graph distance of $x$ from $y$ in $H$, defined as number of edges along of the shortest path from $x$ to $y$ in $H$.
 
\end{lemma}

\begin{lemma}
\label{lem:equivek1}
Let $G$ be an infinite connected  graph of bounded degree and minimal degree at least $2$. Then
\begin{itemize}
\item[(1)] The graphs $G$ and $G_E$ are roughly-isometric.
\item[(2)] SRW on $G_E$ is a factor of SRW on $G_{\dE}$ (in the sense of Definition \ref{def:lumping}).
\item[(3)] There exists a rough-embedding from $G_E$ to  $G_{\dE}$.
\item[(4)] If there exists some $K \in \N$ such that $\mathrm{dist}_{G_{\dE}}(\gamma,\gamma^{\rr}) \le K $ for every $\gamma \in \dE $  then there exists a rough-isometry from $G_E$ to  $G_{\dE}$.
\item[(5)] If condition $(1)$ holds then there exists some $K \in \N$ such that $\mathrm{dist}_{G_{\dE}}(\gamma,\gamma^{\rr}) \le K $ for every $\gamma \in \dE $.  
\end{itemize} 
\end{lemma}
\begin{proof}
For Part (2) observe that $G_E$ is obtained from $G_{\dE} $ by identifying (lumping) every $\de \in \dE$ with its reversal.
 
Fix a sink-free and source-free orientation $f$ (as in Lemma \ref{lem:sinkfree}). For Part (1)  observe that the map $e \to f(e)_+  $ from $E$ to $V$ is clearly a 2-rough-isometry from $G_E$ to $G$ (for this we do not need to use the fact that $f$ is sink-free and source-free). 
For Part (3) we now argue that $f$ is a $2$-rough-embedding from $G_{E}$ to $G_{\dE}$. In fact, we argue that $\mathrm{dist}_{G_E}(e,e') \le \mathrm{dist}_{G_{\dE}}(f(e),f(e')) \le 2 \mathrm{dist}_{G_E}(e,e')$ for all $e,e' \in E$. The first inequality is obvious. For the second inequality it suffices to consider the case that $ \mathrm{dist}_{G_E}(e,e')=1$. 
 If $e \cap e' = \{v\} $ and $ f(e)$ and $f(e') $ are not adjacent in $ G_{\dE}$, then either  $v=f(e)_+=f(e')_+ $ or $v=f(e)_-=f(e')_-$. In the first case, (by condition (ii) in the definition of a sink-free and source-free orientation) there must be some $\{v,v'\} \in E \setminus \{e,e'\} $ such that $f({\{v,v'\}})_{-}=v$ (in the second case $f({\{v,v'\}})_{+}=v $) and so  $f(e) $ and $f(e') $ are both adjacent to $f(\{v,v'\})$  in $ G_{\dE}$.

For Part (4) observe that under the assumption that $\mathrm{dist}_{G_{\dE}}(\gamma,\gamma^{\rr}) \le K $ for every $\gamma \in \dE $,  we have that the map $f$ (defined above) is a rough-isometry. 

For Part (5)
note that if $v \in V$ is of degree at least 3, then similar reasoning as in the proof of Part (iii) yields that for every pair of adjacent vertices $u \sim v$ we have that $(u,v)$ and $(v,u)$ are of distance at most 3 in $G_{\dE}$ (namely, if $w,z$ are other neighbors of $v$ then $((u,v),(v,w),(z,v),(v,u))$ is a path from $(u,v)$ to $(v,u)$ in  $G_{\dE}$).
\end{proof} 
\begin{corollary}
\label{cor:SRWASNBRW}  
Let $G$ be a connected infinite graph of bounded degree. Then SRW on $G$ is transient iff the additive symmetrization of the NBRW on $G$ is transient.
\end{corollary}
\begin{proof}
We shall show  that w.l.o.g.~we may assume that the minimal degree of $G$ is at least 2, and so the assertion of the corollary follows from Lemmas \ref{lem:simpleobservtion} and \ref{lem:equivek1} in conjunction with Fact \ref{f:RIinvariance}. We argue that by repeatedly deleting all degree one vertices, while adding a self-loop at each remaining $v \in V $ such that at least  one  of its neighbors was deleted, we obtain a graph $G'=(V(G'),E(G'))$ of minimal degree at least 2 such that: 
\begin{itemize}
\item[(i)]
SRW on $G$ is transient iff SRW on $G'$ is transient, and also 
\item[(ii)]
the additive symmetrization of the NBRW on $G$ is transient iff the additive symmetrization of the NBRW on $G'$ is transient.
\end{itemize}
While (i) is clear, (ii) requires some justification. Recall the notion of the induced chain on a set from \S\ref{s:Markov}.  We denote the induced chain on the directed edges of $G'$ w.r.t.\ the additive symmetrization of the NBRW on $G$ by $\mathbf{Y}$. Clearly, $\mathbf{Y}$ is transient iff  the additive symmetrization of the NBRW on $G$ is transient, and so it suffices to show that $\mathbf{Y}$ is transient iff $\mathbf{Z}$ is transient, where  $\mathbf{Z}$ is the additive symmetrization of the NBRW on $G'$.  Observe that the graph   $\widehat G=(V(\widehat G),E(\widehat G))$   supporting the transitions of  $\mathbf{Y}$ is  roughly-isometric to $\widehat G'=(V(\widehat G'),E(\widehat G'))$ the graph supporting the transitions of  $\mathbf{Z}$ and that  $\mathbf{Y}$ and  $\mathbf{Z}$ are  both reversible w.r.t.\ the counting measures on their state spaces,  $\pi_{V(\widehat G)} $ and  $\pi_{V(\widehat G')} $, respectively.\footnote{Here we have used the fact that  $\pi_{V(\widehat G)} $   is the restriction to  $V(\widehat G)$   of $\pi_{\dE} $ the counting measure on the directed edges of $G$, and that the additive symmetrization of the NBRW on $G$, whose restriction to $V(\widehat G) $ is $\mathbf{Y}$, is reversible w.r.t.\ $\pi_{\dE} $. (Recall from \S\ref{s:Markov} that for an irreducible Markov chain which is reversible w.r.t. a measure $\pi$, for any recurrent subset $B$ (i.e., one which is a.s.\ visited by the chain) the restriction of $\pi$ to $B$ is reversible for the induced  chain on $B$.)} (For the existence of the claimed rough-isometry, it is crucial that loops were added to $G'$ at each remaining $v \in V $ such that at least  one  of its neighbors was deleted. The rough-isometry can be taken to be the inclusion map from $V(\widehat G) $ to $V(\widehat G') $, which is a $2$-rough-isometry.) 

 The edge-weights corresponding to  $\mathbf{Z}$ are simply its transition probabilities and hence are clearly uniformly bounded from above and from below (as the maximal degree of $G$ is finite; Recall that in general, for a Markov chain  whose transition kernel $P$ is reversible w.r.t.\ a measure $\pi$ the edge-weight of edge $\{x,y\}$ can be taken to be $\pi(x)P(x,y) $)). We claim that the same holds for  $\mathbf{Y}$, which in conjunction with the fact that $\widehat G$ and $\widehat G'$ are roughly-isometric concludes the proof, using Parts (a) and (c) of  Fact \ref{f:RIinvariance}.

Taking $\pi_{V(\widehat G)} $ for the stationary measure of  $\mathbf{Y}$, the corresponding edges weights for   $\mathbf{Y}$   are simply with the transition probabilities. Thus its enough to show that these are uniformly bounded from below (as they are bounded from above by 1). Note that any transition which is also allowed for the NBRW on $G$ has transition probability which is at least half of the inverse of the maximal degree of $G$. Hence it suffices to verify this for transition probabilities from $(b,a)$ to $(v,a)$ and from $ (a,b)$ to $(a,v)$ with $a \in A$ and $b,v \in V(G')$ such that $\{a,u\},\{a,b\} \in E(G') $, where $A$ is the set of vertices of $ V(G')$ to which a self-loop was added in $G'$. In other words, $A$ is the set of vertices of $G$ which are connected in $G$ to a degree one vertex via a path whose internal vertices all have degree two.

To see that such transitions (from $(b,a)$ to $(v,a)$ and from $ (a,b)$ to $(a,v)$ with $a \in A$ and $b,v \in V(G')$ such that $\{a,u\},\{a,b\} \in E(G') $) have probability bounded from below, consider the scenario that the additive symmetrization of the NBRW on $G$ goes from $(b,a)$ to $(a,a')$ for some deleted vertex $a'$ and then goes from $(a,a')$ to $(v,a)$. Similarly, it may move from $(a,b)$ to $(a',a)$ for some deleted vertex $a'$ and then move from $(a',a)$ to $(a,v)$. Here by ``deleted vertex" we mean a vertex belonging to $V \setminus V(G') $.

These scenarios show that the transition probabilities w.r.t.\ $\mathbf{Y} $ from $(b,a)$ to $(v,a)$ and from $ (a,b)$ to $(a,v)$, with $a \in A$ and $b,v \in V(G')$, are indeed uniformly bounded from below, as desired.                 
\end{proof}
We now treat the case $k>1$. Recall the definitions of $\Pke,\Pkv,\Omega_k^{\ee}, \Omega_k^{\vv} ,P_{k,\ee}$ and $P_{k,\vv} $ from \S\ref{s:notation}-\ref{s:MC}. 
\begin{definition}
\label{d:Gk}
Let $G_k^{\ee}:=(\Omega_k^{\ee},E_k^{\ee}) $ and $G_k^{\vv}:=(\Omega_k^{\vv},E_k^{\vv}) $, where $\{a,b \} \in E_k^{\ee}  $ (respectively, $\in E_k^{\vv}  $) if and only if $P_{k,\ee}(a,b)+P_{k,\ee}(b,a)>0 $ (respectively, $P_{k,\vv}(a,b)+P_{k,\vv}(b,a)>0 $).
\end{definition}
  Observe that $G_1^{\ee}=G_1^{\vv}=G_{\dE} $. More generally, under condition  $\ekk$ (respectively, $\vkk$) we have that  $\PP_k^{\ee}=\Omega_k^{\ee} $ (respectively, $\PP_k^{\vv}=\Omega_k^{\vv} $) and we may consider the graph  $H_k^{\ee}:=(U_{k}^{\ee},F_k^{\ee})$  (respectively, $H_k^{\vv}:=(U_{k}^{\vv},F_k^{\vv})$) obtained from $G_k^{\ee} $ (respectively, $G_k^{\vv} $) by identifying each $a \in \Pke$ (respectively, $a \in \Pkv $) with its reversal. More precisely, $U_{k}^{\ee}:=\{\{a,a^{\rr} \}:a \in \Pke  \}$ and $\{\{a,a^{\rr} \},\{b,b^{\rr} \} \} \in F_k^{\ee} $ iff $\sum_{w \in \{a,a^{\rr} \},w' \in \{b,b^{\rr} \}  }P_{\ee,k}(w,w')>0 $. The relations between $G_k^{\ee} $ and $H_k^{\ee} $ and between  $G_k^{\vv} $ and $H_k^{\vv} $ are analogous to the one between $G_{\dE} $ and $G_E$. Similarly, when $G$ is of bounded degree and condition  $\ekk$ (respectively, $\vkk$) holds,   the relation between SRW on  $G_k^{\ee} $ (respectively, $G_k^{\vv} $) and the additive symmetrization of the edge (respectively, vertex) $k$-NBRW on $G$ is analogous to the relation between SRW on $G_{\dE}$ and the additive symmetrization of the NBRW described in Lemma \ref{lem:simpleobservtion}, provided the stationary measure $\hat \pi$ w.r.t.\ we are taking the symmetrization satisfies $\sup_{x,y}\sfrac{\hat \pi(x)}{\hat \pi(y)}<\infty $. For the edge $k$-NBRW this will be proved in Lemma \ref{lem:mik}. For the vertex $k$-NBRW we are only able to prove this under the additional assumption that the graph is vertex-transitive (see the proof of Part (iv) of Theorem \ref{thm:main}).    

 The following lemma is the $k>1$ analog of Lemma \ref{lem:equivek1}. 
\begin{lemma}
\label{lem:equivek}
Let $G$ be a bounded degree graph satisfying condition  $\ekk$ (resp.~$\vkk$).  Then
\begin{itemize}
\item[(1)] The graphs $G$ and $H_k^{\ee}$ (respectively, $H_k^{\vv}$) are roughly-isometric.
\item[(2)] SRW on $H_k^{\ee}$ (respectively, $H_k^{\vv}$) is a factor of SRW on  $G_k^{\ee} $ (respectively, $G_k^{\vv} $).
\item[(3)] If condition  $\ek$ (respectively, $\vk$) holds there exists some $K \in \N$ such that $\mathrm{dist}_{G_k^{\ee}}(\gamma,\gamma^{\rr}) \le K $ for every $\gamma \in \Pke $ (respectively, $\mathrm{dist}_{G_k^{\vv}}(\gamma,\gamma^{\rr}) \le K $ for every $\gamma \in \Pkv$) and so $G_k^{\ee}$ and $H_k^{\ee}$ (respectively, $G_k^{\vv}$ and $H_k^{\vv}$) are roughly-isometric.  
\end{itemize} 
\end{lemma}
\begin{proof}
The proofs of Parts (1) and (2)  are completely analogous to those of Parts (1) and (2) of Lemma \ref{lem:equivek1}  and are thus omitted. The proof of Part (3) is  analogous to that of Parts (3)-(5)  of Lemma \ref{lem:equivek1}  and is thus omitted. (Note that in the proof of Part (3) of Lemma \ref{lem:equivek1} we did not assume condition $(2)$ holds, whereas here we assume condition  $\ek$ (respectively, $\vk$) holds. This makes the proof of Part (3) of Lemma \ref{lem:equivek} easier, by eliminating the need to prove a generalization of Lemma \ref{lem:sinkfree} for $k>1$.)
\end{proof}

\section{An overview of our approach}
\label{s:overview}
In this section as a warm up and as motivation for what comes we prove Theorem \ref{thm:pNBRW}. Recall that $\pi_{\dE} $ is   the counting measure on $\dE$. Let $\dpi:=\pi_{\dE} $. Let $B$ be the transition kernel of the NBRW $(S_n^{\mathrm{NB}})_{n=0}^{\infty}$. While $B$ is non-reversible, it satisfies  a symmetry similar to the detailed balance equation:
\[\forall \, a,b \in \dE, \quad \dpi(a)B(a,b)=\dpi(b^{\rr})B(b^{\rr},a^{\rr})= \dpi(b)B(b^{\rr},a^{\rr}).  \] 
It follows by induction that for every $n \in \N$ and  $a_0,a_1,\ldots ,a_n \in \dE $
\begin{equation}
\label{e:revsymmetry0}
\dpi(a_0)\Pr_{a_{0}}(S_1^{\mathrm{NB}}=a_1,\ldots S_n^{\mathrm{NB}}=a_n ) = \dpi(a_n) \Pr_{a_{n}^{\rr}}(S_1^{\mathrm{NB}}=a_{n-1}^{\rr},\ldots S_n^{\mathrm{NB}}=a_0^{\rr} ).
\end{equation}
While $\dpi $ is redundant above, we include it in order to demonstrate a general principle. By summing over all $a_0,\ldots,a_n$ with $a_0=a$ and $a_n=b$, we get for all $a,b \in \dE$ and $n \in \N $ that
\begin{equation}
\label{e:revsymmetry}
\dpi(a) B^n(a,b) = \dpi(b^{\rr}) B^{n}(b^{\rr},a^{\rr})=\dpi(b) B^{n}(b^{\rr},a^{\rr}). 
\end{equation}
Observe that \eqref{e:revsymmetry} implies that $\dpi$ is stationary for $B$ and so also for $B_p$ (where   $B_p$ is as in  Theorem \ref{thm:pNBRW}). Indeed, by \eqref{e:revsymmetry}  $\sum_b \dpi(b) B(b,a)=\dpi(a)\sum_{b}B(a^{\rr},b^{\rr})=\dpi(a) $. 

  In order to exploit \eqref{e:revsymmetry}  in the proof of Theorem \ref{thm:pNBRW} we consider an auxiliary chain in  which every $a \in \dE$ is identified with $a^{\rr}$. The auxiliary chain is defined by looking at the $p$-BRW at random times at which for every $a \in \dE$  the chain is equally likely to be at $a$ or at $a^{\rr}$. The mechanism which allows us to construct such random times is the fact that the $p$-BRW backtracks at each step with probability $p$. 

From the construction of the auxiliary chain it will be clear that it is recurrent iff the original chain is recurrent.
The main idea is that the auxiliary chain is reversible and thus can be compared with SRW on $G$ via the standard comparison techniques from \S\ref{s:comparison}. That is, it is recurrent iff the SRW on $G$ is recurrent.  

In \S\ref{s:findingpi} we show that in the setups of Theorems  \ref{thm:main} and  \ref{thm:Cayley}  either the   $k$-NBRW or some related auxiliary chain enjoys similar symmetry relations as \eqref{e:revsymmetry0}-\eqref{e:revsymmetry}. Using similar reasoning as above this will allow us to find a stationary measure for the   $k$-NBRW (or in the proof of Theorem \ref{thm:Cayley}, for a related auxiliary chain). We will then look at another auxiliary chain, obtained by looking at the $k$-NBRW (or the aforementioned related chain) at random times at which for every $a  $  it is equally likely to be at $a$ or at $a^{\rr}$. The mechanism which allows us to construct this auxiliary chain is related to condition $\ek$.   Again, crucially, the auxiliary chain is reversible and thus amenable to comparison techniques. 
\begin{remark}
 In the setup of Theorem \ref{thm:main}, analogs of \eqref{e:revsymmetry0}-\eqref{e:revsymmetry} hold for the edge $k$-NBRW (see \eqref{e:pike2}) but not for the vertex $k$-NBRW (when $k>1$).   In general, such symmetry may fail  if the walk may get stuck, because it is possible that a certain trajectory is possible in one direction but not in the opposite direction. In the setup of Theorem \ref{thm:Cayley}  we do not assume that the $k$-NBRW cannot get stuck. To overcome the aforementioned  difficulty we work with a more symmetric auxiliary chain and exploit the fact that the group is Abelian.   
\end{remark}

\emph{Proof of Theorem \ref{thm:pNBRW}:}
Let $G=(V,E)$ and $p \in (0,1)$. Consider the lazy version of $B_p$ whose transition kernel is $\frac{1}{2}(I+B_p)$ (where $I$ is the identity operator). Denote this chain by $\mathbf{Y}:=(Y_t)_{t=0}^{\infty}$. We may think of this chain as  first picking a candidate for its next position according to $B_p$ and then flipping a fair coin to decide whether to accept the candidate or to stay put. Denote the candidate that was picked at time $t$ (for time $t+1$)  by $Z_{t+1}$. 

Let $(x,y) \in \dE $. Consider the case that the initial distribution is uniform on $\{(x,y),(y,x) \}$. Let $\tau_0=-1 $ and $e_0:=\{x,y \} $. Define  inductively $\tau_{n}:=\inf\{t> \tau_{n-1}:Z_{t+1}=Y_t^{\rr} \}$ and $e_n:=\{x_n,y_n \}$, where $x_n$  and $y_n$ are the end-points of $Z_{\tau_n+1} $. In words, $\tau_n$ is the first time after time $\tau_{n-1}$ at which the chain $\mathbf{Y}$ ``attempts" to backtrack (i.e., the candidate directed edge it picked is the reversal of the current directed edge).  Observe that given $(Y_i)_{i=0}^{\tau_n}$ and $(\tau_i)_{i=0}^n$ the conditional distribution of $Y_{\tau_n+1}$ is the uniform distribution on $\{(x_n,y_n ),(y_n,x_n)\}$ (because either $Y_{\tau_n+1}=Z_{\tau_n+1}$ or $Y_{\tau_n+1}=Z_{\tau_n+1}^{\rr}$,  depending on the result of the coin toss).
 
Consider the process  $\mathbf{Q}:=(e_n)_{n=0}^{\infty}$ on $E$. To conclude the proof we show that
\begin{itemize}
\item[(i)] The process $\mathbf{Q}$ is a Markov chain. Its transition kernel $W$ is symmetric (i.e., $W(a,b)=W(b,a) $  for all $a,b\in E$)  and so $\mathbf{Q} $ is  reversible w.r.t.~the counting  measure $\pi_{E}$ on $E$.
\item[(ii)] The $p$-BRW is recurrent iff $\mathbf{Q}$ is recurrent.
\item[(iii)] The SRW on $G$ is recurrent iff $\mathbf{Q}$ is recurrent.

\end{itemize}
  We first prove (i). Let $B_{\mathrm{L}}:=\frac{1}{2}(I+B) $ be the transition kernel of the lazy version of the NBRW. By  \eqref{e:revsymmetry} for every $\alpha,\beta \in \dE $ and $n \in \N$ we have that
\begin{equation}
\label{e:BLsym}
B_{\mathrm{L}}^n(\alpha,\beta)=\sum_{i=0}^n 2^{-n}\binom{n}{i}B^i(\alpha,\beta)=\sum_{i=0}^n 2^{-n}\binom{n}{i}B^i(\beta^{\rr},\alpha^{\rr})=B_{\mathrm{L}}^n(\beta^{\rr},\alpha^{\rr}). \end{equation}
For every $\{x,y \} \in E $ let $D_{\{x,y \}}:=\{(x,y),(y,x) \} $. Using the fact the $Y_{\tau_n+1}$ is uniformly distributed on $D_{e_n} $ we get that for every $a,b \in E$ we have that 
\begin{equation}
\label{e:Wthm3}
W(a ,b)=\frac{1}{2} \sum_{i=0}^{\infty}\sum_{w\in D_{a},w' \in D_b } B_{\mathrm{L}}^i(w,  w')(1-p)^{i}p.  
 \end{equation}
By \eqref{e:BLsym}  $W$ is indeed symmetric. 

We now prove Claim (ii). Since $\mathbf{Y}$ is a lazy version of the $p$-BRW, it is recurrent iff the $p$-BRW is recurrent. Observe that above we constructed a coupling of $\mathbf{Y}$ and $\mathbf{Q} $ (in other words, while  $\mathbf{Q} $ can be defined via \eqref{e:Wthm3}  we constructed it using the process $\mathbf{Y}$). It is easy to see that in this coupling the expected number of (all) returns to the starting point of one chain is within a constant factor from that of the other (since every return of one chain has a constant probability of becoming a return for the other). Hence either both expectations are infinite or both are finite. 

We now prove  (iii). We first prove that w.l.o.g.\ we may assume that $G$ has minimal degree at least 2. By repeatedly deleting degree one vertices we obtain a graph $H'=(V',E')$ with minimal degree two. For every $v \in V'$ let $n(v)$ be the number of degree one vertices of $G$ which are connected to it in $G$ via a path whose internal vertices are all of degree 2.
We then consider the graph $\widehat H=( V',\widehat E)$ obtained from $H'$ by adding $n(v)$  self-loops at each $v \in  V'$.

 Clearly SRW on $G$ is transient iff SRW on $\widehat H$ is transient. We argue that the same is true for the $p$-BRW.  We now argue that for every directed edge $(u,v) $ of $ H'$  we have that the expected number of visits to this edge by the $p$-BRW on $G $ is the same as this expectation for the $p$-BRW on $\widehat H $ (when both $p$-BRWs start at the same directed edge). This is because the two $p$-BRWs can be coupled so that the former enters some path $(v_0,\ldots,v_\ell)$ such that $\deg(v_0)\ge 3 $, $\deg (v_{\ell})=1$ and $\deg (v_i)=2 $ for $i \in [\ell-1]$ and ultimately leaves it from $(v_1,v_0)$ to some $(v_0,u)$ with $u \in V'$ (respectively, to some $(v_0,v_1')$, where  $(v_0,v_{1}'\ldots,v_{\ell'}')$ is such that $\deg (v'_{\ell'})=1$ and $\deg (v_i')=2 $ for $i \in [\ell-1]$) iff the latter crosses the self-loop at $v_0$ corresponding to the path $(v_0,\ldots,v_\ell)$ and then (possibly after ``backtracking" by crossing the same loop for some number of times) crosses to $(v_0,u)$ (respectively, crosses to the self-loop which corresponds to the path   $(v_0,v_{1}'\ldots,v_{\ell'}')$). We can further couple their trajectories together until the next time the former enters such a path, while the latter crosses a corresponding self-loop. Continuing in this fashion we obtain a coupling in which the two $p$-BRWs visit every directed edge $(u,v) $ of $ H'$ the same number of times. In particular, one is recurrent iff the other is recurrent.

Recall that the \emph{line graph} of $G$ is defined as $G_{E}:=(E,E') $, where $\{e,e' \} \in E' $ iff $e \cap e' \neq \eset$ and $e \neq e'$. By Lemma \ref{lem:equivek1} SRW on $G$ is recurrent iff SRW on $G_E$ is recurrent.  
To conclude the proof of Claim (iii) we now show that $\mathbf{Q}$ is recurrent iff SRW on the line graph of $G$ is recurrent. In order to do so we utilize the comparison technique from Lemmas \ref{lem:comparison1} and  \ref{lem:comparison2}. Namely, using the notation and terminology from Definition \ref{d:flow}, by Lemmas \ref{lem:comparison1} and  \ref{lem:comparison2}  it suffices to show that there exists a flow of finite congestion from the network corresponding to  $\mathbf{Q}$  to the network corresponding to SRW on $G_{E}$ and vice-versa.

Let $\tilde E := E' \cup \{ \{e,e\}:e \in E\}$. Denote by $\tilde G:=(E,\tilde E)$ the graph obtained from $G_E$  by adding a self-loop at each site of $G_E$. Below we consider $\tilde G$  instead of $G_E$. Denote the transition kernel of SRW on $\tilde G$ by $U$.  
 The networks corresponding to $U$ and $W$ can be represented, respectively, as $(E,\tilde E,(1)_{w \in \tilde E })$ and $(E,F,(c_{w})_{w \in F})$, where \[F:=\{\{e,e' \}:e,e' \in E \}=\{\{e,e' \}: e,e' \in E \text{ such that } W(e,e')>0 \}\] and $c_{e,e'}:=W(e,e')$ for every $e,e'  \in E$. We start with the easier direction. Observe that there exists a constant $M>0 $ such that $U(e,e')  \le M W(e,e') $ for all $e,e' \in E$.  Thus we may define a $W $-$U $ flow of finite congestion by mapping every $(e,e')$ such that  $\{e,e' \} \in \tilde E $  to itself (with weight 1).

We now construct a $U$-$W$ flow $\Phi$ of finite congestion. Loosely speaking, the flow $\Phi$ assigns to every path the probability that the chain $\mathbf{Y}$ follows that path between two consecutive steps of $\mathbf{Q} $. This is not a precise description because $\mathbf{Y}$ is defined on directed edges, but we need to consider paths in $\tilde G $ (which involve undirected edges). Before defining  $\Phi$   we first introduce some notation. Let \[G':=(\dE,\{\{a,b \}:a,b \in \dE,\, B_{\mathrm{L}}(a,b)>0 \} )\] be the graph supporting the transitions of $B_{\mathrm{L}}$. Denote the collection of all directed paths of length $\ell$ in $G'$ by $\PP_{\ell}(G')$. For every path  $\gamma:=(e_{0},\ldots,e_{\ell}) $ in $\tilde G $ which is not of the form $e_0=\cdots=e_{\ell} $ there is at most one  
 $h(\gamma):=(\overrightarrow{e_{0}},\ldots,\overrightarrow{e_{\ell}}) \in \PP_{\ell}(G')$ (here and below each $f_i  $ is a directed edge, i.e., an element of $\dE$, not a vertex of $G$) such that 
\begin{itemize}
\item[(1)] $\{(\overrightarrow{e_{i}})_{-},(\overrightarrow{e_{i}})_{+}\} = e_{i} $ for all $0 \le i \le \ell, $  \item[(2)] for all $0 \le i < \ell$ if $e_i=e_{i+1} $ then $\overrightarrow{e_{i}}=\overrightarrow{e_{i+1}}$, and
 \item[(3)] for all $0 \le i < \ell$ if $e_i \neq e_{i+1}$ then $(\overrightarrow{e_{i+1}})_-=(\overrightarrow{e_{i}})_+$.  
 \end{itemize}
Denote the collection of all paths   $\gamma=(e_{0},\ldots,e_{\ell}) $ in $\tilde G $ not of the form $e_0=\cdots=e_{\ell} $, which have such $h(\gamma)=(\overrightarrow{e_{0}},\ldots,\overrightarrow{e_{\ell}})  $ as above by $E^{(\ell)} $.
For every path  $\gamma=(e_{0},\ldots,e_{\ell}) \in E^{(\ell)}  $ let
\[\rho(e_{0},\ldots,e_{\ell}):=\frac{(1-p)^{\ell }p}{2}  \prod_{i \in [\ell ] }B_{\mathrm{L}}(\overrightarrow{e_{i-1}},\overrightarrow{e_{i}}), \quad \text{where} \quad h(\gamma)=(\overrightarrow{e_{0}},\ldots,\overrightarrow{e_{\ell}}) . \]

Let $e,e' \in  E$.  For every $e_0=e,e_1,\ldots,e_{\ell}=e' \in  E$ with $ (e_{0},\ldots,e_{\ell}) \in E^{(\ell)}  $  let  $\Phi_{e,e'}((e_0,\ldots,e_{\ell})):=\rho(e_0,\ldots,e_{\ell}) $. It is not hard to verify that by construction we have that $\Phi:=(\Phi_{e,e'})_{\{e,e' \}\in F }$ is a $U$-$W$ flow. It remains only to bound its congestion.

Observe that for all $(e_{0},\ldots,e_{\ell}) \in E^{(\ell)}  $  and all $1 \le i < \ell$ we have that  
\begin{equation}
\label{e:2p1p}
 \rho(e_{0},\ldots,e_{\ell} ) = \frac{2}{p}\rho(e_{0},\ldots,e_{i} ) \rho(e_{i},\ldots,e_{\ell})=\frac{2}{p}\rho(e_{i},\ldots,e_{1},e_{0}) \rho(e_{i},\ldots,e_{\ell}), 
\end{equation}
if $e_0,e_1,\ldots,e_i $ are not all equal and also $e_i,e_1,\ldots,e_{\ell} $ are not all equal. To make this true without this restriction, if $e=e_0=\cdots=e_i$ we define $\rho(e_{0},\ldots,e_{i} ):=\sfrac{(1-p)^{i }p}{2}2^{-i} $. If $(e_0,\ldots,e_i) \notin E^{(i)} $ and is not of the form  $e=e_0=\cdots=e_i$, we set  $\rho(e_{0},\ldots,e_{i} ):=0 $.  

Let  $\{e,e' \} \in \tilde E$.  As in Definition \ref{d:flow} let $A_{e,e'}(\Phi) $ be the congestion of $(e,e')$ w.r.t.~$\Phi$. Finally,   there exists a constant $C_{p}>0$ such that  for every $\{e,e' \} \in \tilde E$ we have that 
\begin{equation*}
\begin{split}
& A_{e,e'}(\Phi)  \le  \sum_{(e_{0},\ldots,e_{\ell} )\in  E^{(\ell)},\, \ell \in \N:\, (e,e')=(e_{i-1},e_i) \text{ for some }i \in [\ell] }\ell^{2} \rho(e_{0},\ldots,e_{\ell})  \\ & \text{(using \eqref{e:2p1p}  }\ell^2 \rho(e_{0},\ldots,e_{\ell})\le \sfrac{8}{p}i^2\rho(e_{i},\ldots,e_{1},e_{0})(\ell-i)^2 \rho(e_{i},\ldots,e_{\ell})) \\ & \text{ for $i>0$ such that }(e_{i-1},e_i)=(e,e')\text{)} 
 \\ & \le \frac{8}{p} \left( \sum_{e_{0},\ldots,e_{\ell} \in  E, \, \ell \in \N :\, ( e_0,e_1)=(e,e') }\ell^{2} \rho(e_{0},\ldots,e_{\ell}) \right)^2 \\ & \text{(writing }q_{\ell}:=\sum_{e_{0},\ldots,e_{\ell} \in  E, \,  :\, ( e_0,e_1)=(e,e') } \rho(e_{0},\ldots,e_{\ell}) \text{ we have that }\sum_{\ell}q_{\ell}=1/2 \\ & \text{ and so }\left(\sum_{\ell}q_{\ell}\ell^2\right)^{2} \le \sum_{\ell}q_{\ell}\ell^4=\sup_{(u,v) \in \dE } \mathbb{E}_{(u,v)}[\tau_1^4]  ) 
\\ & \le \frac{8}{p} \sup_{(u,v) \in \dE } \mathbb{E}_{(u,v)}[\tau_1^4]  \le C_{p}'.
\end{split}
\end{equation*}
This concludes the proof of (iii).    
\qed

\section{Finding a stationary measure for the edge $k$-NBRW}
\label{s:findingpi}
Let $G=(V,E)$ be a graph. Let $s \ge k$. Recall that $\PP_s $ is the collection of paths of length $s$ in $G$.  For $\gamma \in \PP_s $ let
\[m_{i}^{(k)}(\gamma):=|\{j: i-k < j \le i \text{ such that }\gamma_j=\gamma_i \}|. \]
For $\gamma \in \PP_s $ we denote the multi-set $\{m_{i}^{(k)}(\gamma):i \in [s-1] \text{ such that } \gamma_i=v \}$ by $A(v,k,\gamma) $ (it will be crucial in what comes that in this definition we do not consider $i \in \{0,s\}$).

For $\gamma \in \Pke $ let 
\begin{equation}
\label{e:pike}
\pi_{k}^{\ee}(\gamma):=\prod_{i=1}^{k-1} \frac{1}{\deg(\gamma_i)-m_{i}^{(k)}(\gamma)}.
\end{equation}
For   $s \ge  k$  we define
\begin{equation}
\label{Pks}
\begin{split}
\Pke(s)&:=\{ \gamma \in \PP_s: (\gamma_{i},\gamma_{i+1},\ldots,\gamma_{i+k}) \in \Pke \text{ for  all }0 \le i \le s -k \}. 
\end{split}
\end{equation}
We say that two multi-sets are equal if each element appears in both with the same multiplicity.
\begin{lemma}
\label{lem:mik}
Let $k \le s $. Let $\gamma \in \Pke(s)$. Then also $\gamma^{\rr} \in \Pke(s) $. Moreover,  $A(v,k,\gamma)= A(v,k,\gamma^{\rr})  $  (as multi-sets) for all $v $. Consequently, for all  $\gamma \in \Pke $ we have that
\begin{equation}
\label{e:pimdeg}
\pi_{k}^{\ee}(\gamma)=\pi_{k}^{\ee}(\gamma^{\rr}). 
\end{equation}
\end{lemma}
\begin{proof}
The claim that  $\gamma \in \Pke(s)$ iff  $\gamma^{\rr} \in \Pke(s) $ is obvious. It is also easy to see from  \eqref{e:pike} that \eqref{e:pimdeg} is immediate from the claim that   $A(v,k,\gamma)= A(v,k,\gamma^{\rr})  $  for all $v $ and all $\gamma \in \Pke(s)$ (by taking $s=k$). We now sketch the proof of   the claim that    $A(v,k,\gamma)= A(v,k,\gamma^{\rr})  $  for all $v$. First prove the case $k=s$ by induction on $k$. Then, for each fixed $k$ prove the case $s>k$ by induction on $s$ (with the base case being $s=k$). 

In the first induction, in the induction step from $s=k=i$ to $s=k=i+1$,  compare $A(v,i+1,\gamma)$ and  $A(v,i+1,\gamma^{\rr})$ with  $A(v,i,\alpha)$ and  $A(v,i,\alpha^{\rr})$, respectively, where $\alpha$ is obtained from $\gamma$ by omitting from $\gamma$ its last co-ordinate.  In the second induction, in the induction step from $(k,s) $ to $(k,s+1)$ compare $A(v,k,\gamma)$ and  $A(v,k,\gamma^{\rr})$ with  $A(v,k,\alpha)$ and  $A(v,k,\alpha^{\rr})$, respectively, where $\alpha$ is obtained from $\gamma$ by omitting its last co-ordinate.
 \end{proof}
\begin{corollary}
\label{cor:pike}
Assume that the edge $k$-NBRW on $G$ cannot get stuck. Then for all $n \in \N$ and $\alpha,\beta, \alpha_0,\alpha_1,\ldots,\alpha_n \in \Pke  $ we have that
\begin{equation}
\label{e:pike2}
\begin{split}
\pi_{k}^{\ee}(\alpha_0) \prod_{i=0}^{n-1}P_{k,\ee}(\alpha_i,\alpha_{i+1}) &=\pi_{k}^{\ee}(\alpha_n) \prod_{i=0}^{n-1}P_{k,\ee}(\alpha_{n-i}^{\rr},\alpha_{n-(i+1)}^{\rr}).
\\ \pi_{k}^{\ee}(\alpha)P_{k,\ee}^n(\alpha,\beta) &=\pi_{k}^{\ee}(\beta) P_{k,\ee}^n(\beta^{\rr},\alpha^{\rr}).
\end{split}
\end{equation}
In particular, $\pi_{k}^{\ee}$ is stationary for $P_{k,\ee}$.
\end{corollary}
\begin{proof}
The first line of \eqref{e:pike2} is a consequence of Lemma \ref{lem:mik}. Indeed, consider $s=n+k $ and $\gamma:=(\gamma_0,\ldots,\gamma_{n+k})$, where $(\gamma_i,\ldots,\gamma_{i+k})=\alpha_i$ for all $i \in \{0,1,\ldots,n \} $. Then
\begin{equation*}
\begin{split}
& \pi_{k}^{\ee}(\alpha_0) \prod_{i=0}^{n-1}P_{k,\ee}(\alpha_i,\alpha_{i+1})=\prod_{i=1}^{k+n-1} \frac{1}{\deg(\gamma_i)-m_{i}^{(k)}(\gamma)} \\ &  =\prod_{v \in \gamma }\prod_{m \in A(v,k,\gamma)  }\frac{1}{\deg(v)-m}{\stackrel{\text{by Lemma \ref{lem:mik}}}{=}}
  \prod_{v \in \gamma^{\rr} }\prod_{m \in A(v,k,\gamma^{\rr})  }\frac{1}{\deg(v)-m}\\ &=\prod_{i=1}^{k+n-1} \frac{1}{\deg(\gamma_i^{\rr})-m_{i}^{(k)}(\gamma^{\rr})}=\pi_{k}^{\ee}(\alpha_n) \prod_{i=0}^{n-1}P_{k,\ee}(\alpha_{n-i}^{\rr},\alpha_{n-(i+1)}^{\rr}), \end{split}
 \end{equation*}
where each $m  $ in the third (respectively, fourth) term is taken with its multiplicity in the multi-set $A(v,k,\gamma)$ (respectively, $A(v,k,\gamma^{\rr}) $).

 The second line of \eqref{e:pike2} follows from the first by summing over all  $ \alpha_0,\alpha_1,\ldots,\alpha_n \in \Pke  $  such that $\alpha_0=\alpha$ and $\alpha_n=\beta$. Note that here we are using the assumption that the walk cannot get stuck. Finally, $\sum_{\alpha \in \Pke } \pi_{k}^{\ee}(\alpha)P_{k,\ee}(\alpha,\beta)=\pi_{k}^{\ee}(\beta)\sum_{\alpha \in \Pke } P_{k,\ee}(\beta^{\rr},\alpha^{\rr})=\pi_{k}^{\ee}(\beta)$ and so  $\pi_{k}^{\ee}$ is indeed stationary for $P_{k,\ee}$. 
\end{proof}

\section{Proof of Theorem \ref{thm:main}}
\label{s:proofthm1}
We shall need the following simple lemma whose proof is deferred to \S\ref{s:simplelemma}.
\begin{lemma}
\label{lem:simplellemma}
Let $P $ be the transition kernel of an irreducible Markov chain. Let $\hat P:= \sum_{i=0}^{m}p_i P^{i} $ where $\sum_{i=0}^m p_i=1$ and $m \in \N$. If the greatest common denominator (gcd) of $\{i \in [m]:p_i>0 \} $ is 1 then $P$ is recurrent iff $\hat P$ is recurrent.    
\end{lemma}
\subsection{Proof of Parts (i) and (ii) of Theorem \ref{thm:main}}
\label{s:ekNBRW}
\emph{Proof:} The assertion of Part (i) follows from Corollary \ref{cor:SRWASNBRW} in conjunction with Theorem \ref{thm:AG}. We now prove Part (ii). Let $G=(V,E)$ be a connected  graph of bounded degree. Assume that condition $(2)$ holds. Let $R$ be as in condition $(2)$. Let $d$ be the maximal degree in $G$. We first argue that by condition $(2)$ for\footnote{We have not attempted to pick the smallest possible $M$.} $M := 4R+1 $ and $p:= (d-1)^{-M}/[2(M+1) ] $ for every $a \in \dE$ we have that
\begin{equation}
\label{e:Daarr}
\sum_{i=1}^{M}B^i(a,a^{\rr}) \ge 2p(M+1). \end{equation}
Indeed, by  condition $(2)$ for all $a\in \dE$ there is some $i=i(a) \le M$ such that for some $a_0=a,a_1,\ldots,a_i=a^{\rr} $ we have that $B(a_{j-1},a_{j})>0 $ for all $j \in [i] $. Then $\prod_{j \in [i]} B(a_{j-1},a_{j}) \ge (d-1)^{-M}= 2p(M+1)$.
 
Let $D:=\frac{1}{M+1}\sum_{i=0}^{M}B^i$ and $D_{\mathrm{L}}:=\frac{1}{2}(I+D)$. Let $\mathbf{X}:=(X_n)_{n=0 }^{\infty} $ be the Markov chain corresponding to $D_{\mathrm{L}}$. By Lemma \ref{lem:simplellemma} we have that the NBRW is transient iff $\mathbf{X}$ is transient.       

  By \eqref{e:Daarr} we may generate $\mathbf{X}$ as follows: given $X_n=a$ the chain first picks a candidate $Z_{n+1}$ for $X_{n+1}$ according to $D$ (i.e., $Z_{n+1}=b$ with probability $D(a,b)$) and then flips a fair coin in order to decide if the candidate is accepted (i.e., with probability $1/2$ we set $X_{n+1}=Z_{n+1}$ and otherwise we set $X_{n+1}=X_n$). Furthermore, we may generate $\mathbf{X}$ and the $Z_n$'s such that  for some $\xi_0,\xi_1,\ldots$  $\mathrm{i.i.d.}$ Bernoulli($p$) random variables for all $n \ge 0$ we have that $Z_{n+1}=X_n^{\rr}$ whenever $\xi_n=1 $ (and possibly also if  $\xi_n=0 $). Indeed, the probability that $ Z_{n+1}=X_n^{\rr} $, given $X_n $ is at least $p$. Given $X_n$ we may generate $Z_{n+1}$ and $\xi_n $ simultaneously using $U_n \sim U[0,1] $ (where $U[0,1]$ denotes the uniform distribution on $[0,1]$ and $U_1,U_2,\ldots$ are $\mathrm{i.i.d.}$ $U[0,1]$) as follows: $Z_{n+1}$ may taken different values $a_1:=X_n^{\rr} ,a_2,\ldots,a_i$ with probabilities $q_1,\ldots,q_i $ (where the $a_j$'s and the $q_j$'s depend on $X_n$). Set $\xi_n=\Ind{U_n \le p} $. Set $Z_{n+1}=a_i $ if $U_n \in [\sum_{j=1}^{i-1}q_j,\sum_{j=1}^{i}q_j )$ (where $\sum_{j=1}^{0}q_j$ is defined to be 0). Then since regardless of the value of $X_n$ we have that $q_1\ge p$ we get that $Z_{n+1}=a_1$ whenever $\xi_n=1 $.   

Fix some $a=(x_{0},y_{0}) \in \dE$. Consider the case that the initial distribution of $\mathbf{X}$ is the uniform distribution on $\{a,a^{\rr} \}$. Let $e_0=\{x_0,y_0\}$ and $\tau_0=-1$. We define inductively $\tau_{n+1}:=\inf\{t>\tau_n: \xi_{t-1}=1 \} $ and $e_{n}:=\{x_n,y_n\}$, where $x_n$ and $y_n$ are the end-points of $Z_{\tau_n+1} $. Consider the process $\mathbf{Q}:=(e_n)_{n=0 }^{\infty} $ on $E$ given by $Q_n=e_n$ for all $n \ge 0$. To conclude the proof it suffices to show that
\begin{itemize}
\item[(i)] The process $\mathbf{Q}$ is a Markov chain. Its transition kernel $W$ is symmetric. Hence $\mathbf{Q} $ is  reversible w.r.t.~the counting measure $\pi_{E}$ on $E$.
\item[(ii)] The NBRW is recurrent iff $\mathbf{Q}$ is recurrent.
\item[(iii)] SRW on $G$ is recurrent iff $\mathbf{Q}$ is recurrent.
\end{itemize}
We first prove part (i). Consider the transition kernel $K$ on $\dE $ given by
\[K(a,b)=\begin{cases}(1-p)^{-1}D(a,b) & \text{if } b \neq a^{\rr} \\
(1-p)^{-1}(D(a,a^{\rr})-p) & \text{if } b = a^{\rr} \\
\end{cases}. \]
It follows from \eqref{e:revsymmetry} that $K(a,b)=K(b^{\rr},a^{\rr}) $ for all $a,b \in \dE$. Hence  $K_{\mathrm{L}}:=\frac{1}{2}(I+K)$ also satisfies this identity. Finally, the symmetry of $W$ follows from the last identity since   for every $a,b \in E$ we have 
\[W(a ,b):=\frac{1}{2} \sum_{i=0}^{\infty}\sum_{w \in D_{a},w' \in D_b } K_{\mathrm{L}}^i(w, w')(1-p)^{i}p,  \]
where  $D_{\{x,y \}}:=\{(x,y),(y,x) \} $. The proofs of (ii)-(iii) above are almost identical to those of (ii)-(iii) from the proof of Theorem \ref{thm:pNBRW} (and also to that of (iii) in the proof of Part (iii) of Theorem \ref{thm:main} below) and are thus omitted.
 \qed  

\subsection{Proof of Part (iii) of Theorem \ref{thm:main}}

Proof:  Assume that conditions  $\ek$ and $\ekk$ hold. Then there exist some $M \in \N$ and $p \in (0,1) $ such that for every $a \in \Pke$ we have that
\begin{equation}
\label{e:Daarr2}
\sum_{i=1}^{M}P_{k,\ee}^i(a,a^{\rr}) \ge 2p(M+1). \end{equation}
Let $D:=\frac{1}{M+1}\sum_{i=0}^{M}P_{k,\ee}^i $ and $D_{\mathrm{L}}:=\frac{1}{2}(I+D)$. Let $\mathbf{X}:=(X_n)_{n=0 }^{\infty} $ be the Markov chain corresponding to $D_{\mathrm{L}}$. By Lemma \ref{lem:simplellemma} the edge $k$-NBRW is transient iff $\mathbf{X}$ is transient.  By \eqref{e:Daarr2} we may generate $\mathbf{X}$ as follows: given $X_n=a$ the chain first picks a candidate $Z_{n+1}$ for $X_{n+1}$ according to $D$ and then flips a fair coin in order to decide if the candidate is accepted. Furthermore, one may generate $Z_1,Z_{2,}\ldots$ such that for some $\xi_0,\xi_1,\ldots$  $\mathrm{i.i.d.}$ Bernoulli($p$) random variables for all $n \ge 0$ we have that $Z_{n+1}=X_n^{\rr}$ whenever $\xi_n=1 $ (and possibly also if  $\xi_n=0 $). The proof of this claim is identical to the corresponding claim in the proof of Part (ii) of Theorem 1.  

Fix some $a=  \Pke$. Consider the case that the initial distribution of $\mathbf{X}$ is the uniform distribution on $\{a,a^{\rr} \}$. Let $e_0=\{a,a^{\rr} \}$ and $\tau_0=-1$. We define inductively $\tau_{n+1}:=\inf\{t>\tau_n: \xi_{t-1}=1 \} $ and $e_{n}:=\{Z_{\tau_n+1},Z_{\tau_n+1}^{\rr}\}$.  Consider the process  $\mathbf{Q}:=(Q_n)_{n=0 }^{\infty} $ on $U_{k}^{\ee}:=\{\{a,a^{\rr} \}:a \in \Pke  \}$ given by $Q_n:=e_n$ for all $n \ge 0$.  To conclude the proof it suffices to show that
\begin{itemize}
\item[(i)] The process  $\mathbf{Q} $ is a Markov chain. Its transition kernel $W$ satisfies    that for all $a,b \in \Pke $  \[\pi_{k}^{\ee}(a) W(\{a,a^{\rr} \},\{b,b^{\rr} \})=\pi_{k}^{\ee}(b)W(\{b,b^{\rr} \},\{a,a^{\rr} \}), \]  where $\pi_{k}^{\ee}$ is as in \eqref{e:pike}. Hence $\mathbf{Q} $ is  reversible w.r.t.~the  measure $\pi(\{a,a^{\rr} \}):=\pi_{k}^{\ee}(a)$.
\item[(ii)] The edge $k$-NBRW is recurrent iff $\mathbf{Q}$ is recurrent.
\item[(iii)] SRW on $G$ is recurrent iff $\mathbf{Q}$ is recurrent.
\end{itemize}
We first prove part (i). Consider the transition kernel $K$ on $\Pke $ given by
\[K(a,b)=\begin{cases}(1-p)^{-1}D(a,b) & \text{if } b \neq a^{\rr} \\
(1-p)^{-1}(D(a,a^{\rr})-p) & \text{if } b = a^{\rr} \\
\end{cases}. \]
It follows from \eqref{e:pike2} that $K $ is satisfies $\pi_{k}^{\ee}(a)K(a,b)=\pi_{k}^{\ee}(b)K(b^{\rr},a^{\rr}) $ for all $a,b \in \Pke$. Hence also $K_{\mathrm{L}}:=\frac{1}{2}(I+K)$ satisfies  $\pi_{k}^{\ee}(a)K_{\mathrm{L}}(a,b)=\pi_{k}^{\ee}(b)K_{\mathrm{L}}(b^{\rr},a^{\rr}) $ for all $a,b \in \Pke$. The assertion of (i) now follows from the fact that for every $a,b \in \Pke$ we have that 
\[W(\{a,a^{\rr} \} ,\{b,b^{\rr} \}):=\frac{1}{2} \sum_{i=0}^{\infty}\sum_{w\in\{a,a^{\rr} \},w' \in \{b,b^{\rr} \}  } K_{\mathrm{L}}^i(w, w')(1-p)^{i}p.  \]
The proof of (ii) is analogous of that of claim (ii) from the proof of Theorem \ref{thm:pNBRW}, and thus omitted.

We now prove  (iii). Recall that   $H_k^{\ee}:=(U_{k}^{\ee},F_k^{\ee})$  is the graph obtained from $G_k^{\ee} $ (from Definition \ref{d:Gk}) by identifying each $a \in \Pke$ with its reversal. More precisely, (as above) $U_{k}^{\ee}:=\{\{a,a^{\rr} \}:a \in \Pke  \}$ and $\{\{a,a^{\rr} \},\{b,b^{\rr} \} \} \in F_k^{\ee} $ iff $\sum_{w \in \{a,a^{\rr} \},w' \in \{b,b^{\rr} \}  }P_{\ee,k}(w,w')>0 $. Let $U$ be the transition kernel  of SRW on  $\tilde H:=(U_{k}^{\ee},\tilde F)$, where \[\tilde F:= \left\{ \left\{\{a,a^{\rr} \},\{b,b^{\rr} \right\} \}:a,b \in \PP_{k}^{\ee} \text{ such that }\sum_{w \in \{a,a^{\rr} \},w' \in \{b,b^{\rr} \}  }D_{\mathrm{L}}(w,w')  >0 \right\}.\]
 By Lemma \ref{lem:equivek} the SRW on $G$ is recurrent iff the SRW on $ H_k^{\ee}$ is recurrent. Since  $ H_k^{\ee}$ and $\tilde H $ are roughly-isometric, to conclude the proof of Claim (iii) we now show that $\mathbf{Q}$ is recurrent iff the SRW on $\tilde H$ is recurrent. By Lemmas \ref{lem:comparison1} and  \ref{lem:comparison2}  it suffices to show that there exist a $U$-$W$ flow and a $W$-$U$ flow of finite congestion. Similarly to the proof of Claim (iii) in the proof of Theorem \ref{thm:pNBRW}, the  $W$-$U$ flow obtained by mapping every $(e,e')$ such that  $\{e,e' \} \in \tilde F $  to itself (with weight 1) is of finite congestion.

We now construct a $U$-$W$ flow $\Phi$ of finite congestion. Loosely speaking, the flow $\Phi$ assigns to every path the probability that the chain $\mathbf{X}$ follows that path between two consecutive steps of $\mathbf{Q} $. This is not a precise description because $\mathbf{X}$ is defined on $\Pke$ rather than on $U_{k}^{\ee}$. Before defining  $\Phi$   we first introduce some notation. Let $\tilde G =(\Pke,\tilde E)$ be the graph supporting the transitions of $D_{\mathrm{L}}$ (i.e.~$\{a,b\} \in \tilde E $ iff $D_{\mathrm{L}}(a,b)>0$). Denote the collection of all directed paths of length $\ell$ in $\tilde G$ by $\PP_{\ell}(\tilde G)$. For every path $(a_0,\ldots,a_{\ell}) \in \PP_{\ell}(\tilde G) $ let \[\sigma(a_0,\ldots,a_{\ell}):=\prod_{i \in [\ell] }K_{\mathrm{L}}^i(a_{i-1},a_i). \]

For every $e_{0},\ldots,e_{\ell}\in U_{k}^{\ee} $ let \[J(e_{0},\ldots,e_{\ell}):=\{(a_0,\ldots,a_{\ell}) \in \PP_{\ell}(\tilde G):\{a_i,a_{i}^{\rr}\} = e_{i} \text{ for all }0 \le i \le \ell \}, \]
\[\rho(e_{0},\ldots,e_{\ell})=\frac{\pi(e_0)}{2} \sum_{(a_0,\ldots,a_{\ell}) \in J(e_{0},\ldots,e_{\ell})}\sigma(a_0,\ldots,a_{\ell })(1-p)^{\ell }p. \]
As oppose to the situation in the proof of Theorem \ref{thm:pNBRW}, if $J(e_{0},\ldots,e_{\ell}) $ is not empty, then it may contain more than one path even if it is not the case that $e_0=\cdots = e_{\ell}$. If $J(e_{0},\ldots,e_{\ell})=\eset $ then we set $\rho(e_{0},\ldots,e_{\ell})=0 $.

Let $e,e' \in  U_{k}^{\ee} $ be such that $W(e,e')>0$. For every $e_0=e,e_1,\ldots,e_{\ell}=e' \in  U_{k}^{\ee} $ let  $\Phi_{e,e'}((e_0,\ldots,e_{\ell})):=\rho(e_0,\ldots,e_{\ell}) $. It is not hard to verify that by construction we have that $\Phi:=(\Phi_{e,e'})_{e,e' :W(e,e')>0 }$ is a $W$-$U$ flow. It remains only to bound its congestion.

Observe that for all every $e_{0},\ldots,e_{\ell} \in  U_{k}^{\ee}  $ and $0 \le i \le \ell$ we have that\footnote{As opposed to the situation in the proof of Theorem \ref{thm:pNBRW}, the first inequality is no longer an equality. E.g., the left most term can be 0 while the middle term is positive. Also, it is possible that middle term may contain contributions coming from opposite orientations $e_i$.}  
\begin{equation*}
 \rho(e_{0},\ldots,e_{\ell} ) \le \frac{2 }{p \pi(e_i)}\rho(e_{0},\ldots,e_{i} ) \rho(e_{i},\ldots,e_{\ell}) \le C_0  \rho(e_{i},\ldots,e_{1},e_{0}) \rho(e_{i},\ldots,e_{\ell}).
\end{equation*}
The remainder of the proof is analogous to the remainder of the proof  of Theorem \ref{thm:pNBRW} after \eqref{e:2p1p}. \qed

\subsection{Proof of Part (iv) of Theorem \ref{thm:main}}
\label{s:auxvNBRW}
\begin{proof} Let  $G=(V,E)$ be a vertex-transitive  graph satisfying conditions  $\vk$ and $\vkk$. Assume that the SRW on $G$ is transient. Assume towards a contradiction that the vertex $k$-NBRW is recurrent. Since $G$ is transient, by Part (3) of Lemma \ref{lem:equivek} the SRW on $G_{k}^{\vv}$ is also transient. We shall show that $\sup_{\alpha,\beta \in \Pkv }\frac{\pi(\alpha)}{\pi(\beta)}< \infty $. Using Theorem \ref{thm:AG} and Fact \ref{f:RIinvariance} we get that the vertex $k$-NBRW is transient, in contradiction to our assumption.

By conditions  $\vk$ and $\vkk$ the vertex $k$-NBRW is irreducible. By recurrence, its stationary measure $\pi$ is unique up to a multiplication by a constant factor. Moreover, by irreducibility and recurrence, for all $\alpha,\beta \in \Pkv $ we have that \cite{Derman} \[\sfrac{\pi(\alpha)}{\pi(\beta)}=\lim_{t \to \infty} \sfrac{\sum_{i=0}^t P_{k,\vv}^{i}(\alpha,\alpha)}{\sum_{i=0}^t P_{k,\vv}^{i}(\beta,\beta)} .\] 

Every automorphism $\phi$ of $G$ can be extended into a probability preserving bijection of $\Pkv $,  by defining  $ \phi(\alpha):=(\phi (\alpha_0),\ldots,\phi(\alpha_k) )$ for $\alpha=(\alpha_0,\ldots,\alpha_k) \in  \Pkv $. It is `probability preserving' in the sense that $P_{k,\vv}^{i}(\alpha,\beta)=P_{k,\vv}^{i}(\phi(\alpha),\phi(\beta)) $ for all $\alpha,\beta \in \Pkv$ and $i \in \N$.  In particular, 
\begin{equation}
\label{e:piapiphia}
\sfrac{\pi(\alpha)}{\pi(\phi(\alpha))}=\lim_{t \to \infty} \sfrac{\sum_{i=0}^t P_{k,\vv}^{i}(\alpha,\alpha)}{\sum_{i=0}^t P_{k,\vv}^{i}(\beta,\beta)}=1. \end{equation}   
 
  By conditions  $\vk$ and $\vkk$ there exist some $M \in \N$ and $p>0$ such that if $\alpha,\beta \in \Pkv $ and $\alpha_0=\beta_0$ then $\sum_{i=1}^M P_{k,\vv}^i(\alpha,\beta)>p $ and consequently for $c=p/M$ we have that 
\begin{equation}
\label{e:pibetacpialpha}
\pi(\beta) = \frac{1}{M}\sum_{i=1}^M \sum_{\gamma \in \Pkv} \pi(\gamma) P_{k,\vv}^i(\gamma,\beta) \ge \frac{1}{M}\sum_{i=1}^M  \pi(\alpha)P_{k,\vv}^i(\alpha,\beta) \ge c \pi(\alpha) .
\end{equation}
 Finally, by \eqref{e:piapiphia} and the fact that $G$ is vertex-transitive, $\sup_{\alpha,\beta \in \Pkv }\frac{\pi(\alpha)}{\pi(\beta)}=\sup_{\alpha,\beta \in \Pkv,\, \alpha_0=\beta_0 }\frac{\pi(\alpha)}{\pi(\beta)}$, and so by \eqref{e:pibetacpialpha} $\sup_{\alpha,\beta \in \Pkv }\frac{\pi(\alpha)}{\pi(\beta)} \le 1/c $, as desired. 
\end{proof}

\section{Proof of Theorem \ref{thm:Cayley}}
\label{s:proofthm2}

Let $G=H(S)$ be as in Theorem \ref{thm:Cayley}. As $H$ is Abelian, we denote the group's operation by `$+$' and denote the inverse of $h \in H$ by $-h$. For  $h \in H$  and $k \in \N$ we define $kh $ inductively to be $h+(k-1)h$. We denote the identity element by $0$. Recall the notion of an induced chain from \S\ref{s:Markov}.   To exploit the symmetry of $G$ we consider the following auxiliary chain.  
\begin{definition}
\label{d:alphah} Let $G=H(S)$ be the Cayley graph of an infinite finitely generated Abelian group $H \neq \Z $ w.r.t.~a finite symmetric set of generators $S$. Fix some $s \in S$ of infinite order (there must be such $s \in S$).  Let $k \in \N$.
 For every $h \in H $ let \[\alpha_h:=(h,h+s,\ldots,h+ks).  \]
Let $\mathcal{A}:=\{\alpha_h,\alpha_h^{\rr} : h \in H \}$. Let $\pi_{\mathcal{A}} \equiv 1 $ be the counting measure on $\mathcal{A} $. 
Let $\mathbf{Y}^{\ee}:=(Y_t^{\ee})_{t=0}^{\infty}$ (respectively, $\mathbf{Y}^{\vv}:=(Y_t^{\vv})_{t=0}^{\infty}$) be the induced chain on $\mathcal{A}$ w.r.t.~the edge (respectively, vertex) $k$-NBRW. Denote its transition matrix by $P_{\mathcal{A},\ee}$ (respectively, $P_{\mathcal{A},\vv}$).
\end{definition}

\begin{lemma}
\label{lem:Csym}
In the setup of Definition \ref{d:alphah} we have that
 $\pi_{\mathcal{A}}$ is stationary for both $\mathbf{Y}^{\ee} $ and $\mathbf{Y}^{\vv} $. Moreover,   for all $a,b \in \mathcal{A} $ we have that
 \begin{equation}
 \label{e:Cayleysymm}
 P_{\mathcal{A},\ee}(a,b)=P_{\mathcal{A},\ee}(b^{\rr},a^{\rr}) \quad \text{and} \quad P_{\mathcal{A},\vv}(a,b)=P_{\mathcal{A},\vv}(b^{\rr},a^{\rr}).
 \end{equation} 
\end{lemma}
We defer the proof to \S\ref{s:simplelem2}. The main reason for which the assertion of the lemma holds is that Cayley graphs of finitely generated Abelian groups satisfy that for every pair of vertices $x,y$ there is an automorphism that maps $x$ to $y$ and $y$ to $x$. Namely, first translate by $-x$, then map every $a$ to $-a$ and finally translate by $+y$.   A more precise explanation is that Lemma \ref{lem:Csym}   is a consequence of  symmetry under negation, along with translation invariance and the fact that $\mathcal{A}$ is closed under negation, i.e., $-\mathcal{A}:=\{-\gamma :\gamma\in \mathcal{A} \}=\mathcal{A} $, where for $\gamma=(\gamma_0,\ldots,\gamma_k) $ we define $-\gamma:= (-\gamma_0,\ldots,-\gamma_k)$. By `symmetry under negation' (respectively, `translation invariance') we mean that for the $k$-NBRW the transition probability from $\gamma$ to  $\tilde \gamma$ equals its  transition probability from $-\gamma$ to  $-\tilde \gamma$ (respectively, from $\gamma+h:= (\gamma_0+h,\ldots,\gamma_k+h) $ to $\tilde \gamma +h$, for all $h \in H$).   

\emph{Proof of Theorem \ref{thm:Cayley}:} We only prove the equivalence between the SRW and the vertex $k$-NBRW as the analysis of the edge $k$-NBRW is analogous. Using the notation from Definition \ref{d:alphah}   there exists some $p \in (0,1) $ such that for every $a \in \mathcal{A}$ we have that
\begin{equation}
\label{e:Daarr4}
P_{\mathcal{A},\vv}(a,a^{\rr}) = p 
\end{equation}
(the constant $p$ is independent of $a$ by the aforementioned translation invariance and symmetry under negation)
Let $\mathbf{X}:=(X_n)_{n=0}^{\infty} $ be the Markov chain corresponding to the transition kernel $P':=\frac{1}{2}(I+ P_{\mathcal{A},\vv}) $.  Since $\mathcal{A}$ is a.s.\ visited by the vertex $k$-NBRW we have that the vertex $k$-NBRW is transient iff $P_{\mathcal{A},\vv} $ is transient, which occurs iff $\mathbf{X}$ is transient. 

By \eqref{e:Daarr4} we may generate $\mathbf{X}$ as follows: given $X_n=a$ the chain first picks a candidate $Z_{n+1}$ for $X_{n+1}$ according to $P'$ and then flips a fair coin in order to decide if the candidate is accepted. Furthermore, one may generate $Z_1,Z_{2,}\ldots$ such that for some $\xi_0,\xi_1,\ldots$  $\mathrm{i.i.d.}$ Bernoulli($p$) random variables,  for all $n \ge 0$ we have that $Z_{n+1}=X_n^{\rr}$ iff $\xi_n=1$. 

Fix some $a \in  \mathcal{A}$. Consider the case that the initial distribution of $\mathbf{X}$ is the uniform distribution on $\{a,a^{\rr} \}$. Let $e_0=\{a,a^{\rr} \}$ and $\tau_0=-1$. We define inductively $\tau_{n+1}:=\inf\{t>\tau_n: \xi_{t-1}=1 \} $ and $e_{n}:=\{Z_{\tau_n+1},Z_{\tau_n+1}^{\rr}\}$.  Consider the process  $\mathbf{Q}:=(e_n)_{n=0 }^{\infty} $ on $\mathcal{B} :=\{\{a,a^{\rr} \}:a \in \mathcal{A}\}$.  To conclude the proof it suffices to show that
\begin{itemize}
\item[(i)] The process  $\mathbf{Q} $ is a Markov chain. Its transition kernel $W$ is symmetric. Hence $\mathbf{Q} $ is  reversible w.r.t.\ the counting measure on $\mathcal{B}$.
\item[(ii)] The vertex $k$-NBRW is recurrent iff $\mathbf{Q}$ is recurrent.
\item[(iii)] SRW on $G$ is recurrent iff $\mathbf{Q}$ is recurrent.
\end{itemize}
We first prove part (i). Consider the transition kernel $K$ on $ \mathcal{A} $ given by
\[K(a,b):=\begin{cases}(1-p)^{-1}P_{\mathcal{A},\vv}(a,b) & \text{if } b \neq a^{\rr} \\
(1-p)^{-1}(P_{\mathcal{A},\vv}(a,a^{\rr})-p) & \text{if } b = a^{\rr} \\
\end{cases}. \]
It follows from \eqref{e:Cayleysymm} that $K(a,b)=K(b^{\rr},a^{\rr})$ for all $a,b \in \mathcal{A}$. Hence $K_{\mathrm{L}}:=\frac{1}{2}(I+K)$ satisfies this identity as well. The assertion of (i) now follows from the fact that for every $a,b \in\mathcal{A}$ we have that 
\[W(\{a,a^{\rr} \} ,\{b,b^{\rr} \}):=\frac{1}{2} \sum_{i=0}^{\infty}\sum_{w\in\{a,a^{\rr} \},w' \in \{b,b^{\rr} \} } K_{\mathrm{L}}^i(w, w')(1-p)^{i}p.  \]
The proofs of (ii) is analogous to that of claim (ii) from the proof of Theorem \ref{thm:pNBRW} and is thus omitted. We now prove (iii). Let $U$ be the transition matrix of SRW on $G$. Observe that we may identify  $\mathcal{B}$ with $H$ via the map $\{\alpha_{h},\alpha_{h}^{\rr}\} \to h $. Hence it suffices to construct a $U$-$W$ flow and a $W$-$U$ flow of finite congestion. The details are similar to those from the proof of Part (iii) of Theorem \ref{thm:main} and are thus omitted.   \qed 

\section*{Acknowledgements}
The author is grateful to Itai Benjamini, Tom Hutchcroft, Yuval Peres and Ariel Yadin  for useful discussions. The author would like to thank Yuval Peres for pointing out the fact that Proposition \ref{prop:reg} is a simple consequence of the well-known coupling used in its proof.

\bibliographystyle{plain}
\bibliography{NBW}

\begin{thebibliography}{10}

\bibitem{aldous}
David Aldous and Jim Fill.
\newblock {\em Reversible {M}arkov chains and random walks on graphs}.
\newblock Unfinished manuscript.

\bibitem{alon}
Noga Alon, Itai Benjamini, Eyal Lubetzky, and Sasha Sodin.
\newblock Non-backtracking random walks mix faster.
\newblock {\em Commun. Contemp. Math.}, 9(4):585--603, 2007.

\bibitem{benhamu}
Anna Ben-Hamou and Justin Salez.
\newblock Cutoff for nonbacktracking random walks on sparse random graphs.
\newblock {\em Ann. Probab.}, 45(3):1752--1770, 2017.

\bibitem{beres}
Nathana\"{e}l Berestycki, Eyal Lubetzky, Yuval Peres, and Allan Sly.
\newblock Random walks on the random graph.
\newblock {\em Ann. Probab.}, 46(1):456--490, 2018.

\bibitem{bordenave}
Charles Bordenave, Marc Lelarge, and Laurent Massouli\'{e}.
\newblock Non-backtracking spectrum of random graphs: community detection and
  non-regular {R}amanujan graphs.
\newblock In {\em 2015 {IEEE} 56th {A}nnual {S}ymposium on {F}oundations of
  {C}omputer {S}cience---{FOCS} 2015}, pages 1347--1357. IEEE Computer Soc.,
  Los Alamitos, CA, 2015.

\bibitem{Derman}
C.~Derman.
\newblock A solution to a set of fundamental equations in {M}arkov chains.
\newblock {\em Proc. Amer. Math. Soc.}, 5:332--334, 1954.

\bibitem{comparison}
Persi Diaconis and Laurent Saloff-Coste.
\newblock Comparison theorems for reversible {M}arkov chains.
\newblock {\em Ann. Appl. Probab.}, 3(3):696--730, 1993.

\bibitem{doyle}
Peter~G. Doyle and J.~Laurie Snell.
\newblock {\em Random walks and electric networks}, volume~22 of {\em Carus
  Mathematical Monographs}.
\newblock Mathematical Association of America, Washington, DC, 1984.

\bibitem{doyle2}
Peter~G Doyle and Jean Steiner.
\newblock Commuting time geometry of ergodic markov chains.
\newblock {\em arXiv preprint arXiv:1107.2612}, 2011.

\bibitem{comparison2}
Martin Dyer, Leslie~Ann Goldberg, Mark Jerrum, and Russell Martin.
\newblock Markov chain comparison.
\newblock {\em Probab. Surv.}, 3:89--111, 2006.

\bibitem{CLT}
Robert Fitzner and Remco van~der Hofstad.
\newblock Non-backtracking random walk.
\newblock {\em J. Stat. Phys.}, 150(2):264--284, 2013.

\bibitem{nonrev}
A.~Gaudilli\`ere and C.~Landim.
\newblock A {D}irichlet principle for non reversible {M}arkov chains and some
  recurrence theorems.
\newblock {\em Probab. Theory Related Fields}, 158(1-2):55--89, 2014.

\bibitem{kempton}
Mark Kempton.
\newblock A non-backtracking {P}\'{o}lya's theorem.
\newblock {\em J. Comb.}, 9(2):327--343, 2018.

\bibitem{lee}
Chul-Ho Lee, Xin Xu, and Do~Young Eun.
\newblock Beyond random walk and metropolis-hastings samplers: why you should
  not backtrack for unbiased graph sampling.
\newblock In {\em ACM SIGMETRICS Performance evaluation review}, volume~40,
  pages 319--330. ACM, 2012.

\bibitem{levin}
David~A. Levin and Yuval Peres.
\newblock {\em Markov chains and mixing times}.
\newblock American Mathematical Society, Providence, RI, 2017.
\newblock Second edition of [MR2466937], With contributions by Elizabeth L.
  Wilmer and a chapter on ``Coupling from the past'' by James G. Propp and
  David B. Wilson.

\bibitem{ram}
Eyal Lubetzky and Yuval Peres.
\newblock Cutoff on all {R}amanujan graphs.
\newblock {\em Geom. Funct. Anal.}, 26(4):1190--1216, 2016.

\bibitem{GW}
Russell Lyons, Robin Pemantle, and Yuval Peres.
\newblock Ergodic theory on {G}alton-{W}atson trees: speed of random walk and
  dimension of harmonic measure.
\newblock {\em Ergodic Theory Dynam. Systems}, 15(3):593--619, 1995.

\bibitem{lyons}
Russell Lyons and Yuval Peres.
\newblock {\em Probability on trees and networks}, volume~42 of {\em Cambridge
  Series in Statistical and Probabilistic Mathematics}.
\newblock Cambridge University Press, New York, 2016.

\bibitem{cogrowth2}
S.~Northshield.
\newblock Cogrowth of regular graphs.
\newblock {\em Proc. Amer. Math. Soc.}, 116(1):203--205, 1992.

\bibitem{ortner}
Ronald Ortner and Wolfgang Woess.
\newblock Non-backtracking random walks and cogrowth of graphs.
\newblock {\em Canad. J. Math.}, 59(4):828--844, 2007.

\end{thebibliography}

\vspace{2mm}

\appendix

\section{Technical Proofs}

\subsection{Proof of Proposition \ref{prop:reg}}
\label{s:coupling}
\begin{proposition*}
Let $G=(V,E)$ be a $d$-regular connected graph for some $d \ge 3$. Then the SRW on $G$ is recurrent iff the NBRW on $G$ is recurrent.
\end{proposition*}  
\emph{Proof.} Fix some $o \in V$. Let $T$ be the universal cover of $G$. Namely, $T=(V',E')$ is an infinite $d$-ary tree rooted at $o$, whose $\ell$th level is labelled by the collection of all non-backtracking paths in $G$ of length $\ell$ started from $o$. The children of a vertex $\gamma=(o,v_1,\ldots,v_{\ell})$ in $T$ are precisely the vertices which are labelled by the possible extensions of $\gamma$ to a non-backtracking path of length $\ell+1$ (i.e.~all paths of the form $(o,v_1,\ldots,v_{\ell+1})$ with $v_{\ell+1} \neq v_{\ell-1} $ and $\{v_{\ell},v_{\ell+1}\} \in E $). For every $\gamma=(o,v_1,\ldots,v_{\ell}) \in V' $ let $\psi(\gamma)=v_{\ell}$.
 
Let $\mathbf{Y}:=(Y_i)_{i=0}^{\infty}$ (respectively, $\mathbf{W}:=(W_i)_{i=0}^{\infty}$) be a SRW (respectively, NBRW) on $T$ started from $o$. Observe that  $(\psi(Y_i))_{i=0}^{\infty}$ (respectively, $(\psi( W_i))_{i=0}^{\infty}$) is a SRW (respectively, NBRW) on $G$ started from $o$. Clearly, the SRW (respectively, NBRW) on $G$ is recurrent iff $\mathrm{a.s.}$ $|\{i: \psi(Y_i)=o \}|=\infty$ (respectively, $|\{i: \psi(W_i)=o \}|=\infty$). We may generate $\mathbf{W}$ from $\mathbf{Y}$ simply as the ray along which  $\mathbf{Y}$  diverges (that is, $W_i$ is the last vertex in the $i$th level of $T$ to be visited by  $\mathbf{Y}$). As $\{i: \psi(W_i)=o \} \subseteq \{i: \psi(Y_i)=o \}$, we get that if the NBRW on $G$ is recurrent then so is the SRW. 

\medskip

Conversely, assume that the SRW on $G$ is recurrent. We will show that every time at which the SRW on $T$ reaches some $v \in V' $ with $\psi(v)=o$, with probability (uniformly) bounded from below $v$ belongs to the ray through which the walk diverges.

 Denote the $i$th level of $T$ by $\mathcal{L}_i$. Let $\sigma_0=0=\rho_0$ and define inductively for all $i \ge 0$, \[\tau_{i}:=\sup\{t  : Y_{t} \in \mathcal{L}_{\rho_i}  \}, \quad  \sigma_{i+1}:=\inf\{t> \tau_i:\psi(Y_t)=o\} \text{ and }\rho_{i+1}:=\mathrm{dist}_T(Y_{\sigma_{i+1}},o). \]  It is not hard to see that $\mathrm{a.s.}$ $\sigma_i$ is finite for all $i$ (by the condition that the SRW is recurrent, and the fact that every level of $T$ is $\mathrm{a.s.}$ visited only finitely many times) and that for all $i$, conditioned on $((\tau_j,\sigma_{j+1},\rho_{j+1}))_{j=0}^{i-1}$, the probability that $\tau_{i}= \sigma_{i} $ is at least $\frac{d-2}{d-1} $. Thus $\mathrm{a.s.}$  there are infinitely many $i$'s such that $\tau_i=\sigma_i$. Note that for such $i$'s we have that $\psi(Y_{\sigma_i})=o$ and $Y_{\sigma_i} \in \{ \psi(W_n):n \in \Z_+ \}$.  \qed

\subsection{Existence of a sink-free and source-free orientation - Proof of Lemma \ref{lem:sinkfree}}
\label{s:sinkfree}
\begin{lemma*}
Let $G=(V,E)$ be an infinite locally finite connected  graph of minimal degree at least $2$. There exists a \emph{sink-free and source-free orientation} of $E$.
\end{lemma*}
\emph{Proof.}
 Label $V$ by $\N$. We start by picking an infinite bi-infinite path $\gamma^{(1)}=(\gamma_i^{(1)}:i \in \Z)$ containing the vertex $k_1$ with label $1$ and pick the orientation on it so that $\{ \gamma_i^{(1)},\gamma_{i+1}^{(1)} \} $  gets the orientation $( \gamma_i^{(1)},\gamma_{i+1}^{(1)} ) $ for all $i$ in $\Z$. Let $k_2$ be the vertex with the minimal label which is not part of $\gamma^{(1)}$. If there exists a finite path $\gamma^{(2)}=(\gamma_0^{(2)},\ldots,\gamma_\ell^{(2)}) $ with $k_2 \in \{\gamma_j^{(2)}:j \in [ \ell -1] \} \subseteq V \setminus \{\gamma_i^{(1)} : i \in \Z \} $ and  $\gamma_0^{(2)},\gamma_{\ell}^{(2)} $ in $\gamma^{(1)}$ then pick one such path and pick the orientation on it so that $\{ \gamma_{i-1}^{(2)},\gamma_{i}^{(2)} \} $  gets the orientation $( \gamma_{i-1}^{(2)},\gamma_{i}^{(2)} ) $ for all $i \in [\ell] $. If no such path exists, then there exists an infinite path  $\gamma^{(2)}=(\gamma_0^{(2)},\gamma_1^{(2)},\ldots) $ with $k_2 \in \{\gamma_j^{(2)}:j \in \N  \} \subseteq V \setminus \{\gamma_i^{(1)} : i \in \Z \} $ and  $\gamma_0^{(2)}$ in $\gamma$. Pick one such path and pick the orientation on it so that $\{ \gamma_{i-1}^{(2)},\gamma_{i}^{(2)} \} $  gets the orientation $( \gamma_{i-1}^{(2)},\gamma_{i}^{(2)} ) $ for all $i \in \N $. 

Set $A_m $ to be the collection of all vertices belonging to the union  of the  first $m$ paths of the construction. If $A_m=V $ set $A_i=A_m$ for all $i \ge m$. Otherwise,
 at stage $m+1$ let $k_{m+1}$ be  the vertex with the minimal label which is not part of $A_m$. If there exists a finite path $\gamma^{(m+1)}=(\gamma_0^{(m+1)},\ldots,\gamma_\ell^{(m+1)}) $ with $k_{m+1} \in \{\gamma_j^{(m+1)}:j \in [ \ell -1] \} \subseteq V \setminus A_{m} $ and  $\gamma_0^{(m+1)},\gamma_{\ell}^{(m+1)} $ in $A_m $ then pick one such path and pick the orientation on it so that $\{ \gamma_{i-1}^{(m+1)},\gamma_{i}^{(m+1)} \} $  gets the orientation $( \gamma_{i-1}^{(m+1)},\gamma_{i}^{(m+1)} ) $ for all $i \in [\ell] $. If no such path exists, then there exists an infinite path  $\gamma^{(m+1)}=(\gamma_0^{(m+1)},\gamma_1^{(m+1)},\ldots) $ with $k_{m+1} \in \{\gamma_j^{(m+1)}:j \in \N  \} \subseteq V \setminus A_{m} $ and  $\gamma_0^{(m+1)}$ in $ A_{m}$. Pick one such path and pick the orientation on it so that $\{ \gamma_{i-1}^{(m+1)},\gamma_{i}^{(m+1)} \} $  gets the orientation $( \gamma_{i-1}^{(m+1)},\gamma_{i}^{(m+1)} ) $ for all $i \in \N $. Note that $A_{\infty}:=\cup_{i \in \N }A_i =V $ and that by construction every vertex has at least one edge oriented towards it and one away from it, as can be seen by considering the stage $m$ in which the vertex first belonged to $A_m$, noting that the vertex must have been an internal vertex of the path $\gamma^{(m)} $. Finally, the construction is concluded by picking an arbitrary orientation for the remaining edges. 
\qed

\subsection{Proof of Lemma \ref{lem:simplellemma}}
\label{s:simplelemma}
\begin{lemma*}
Let $P $ be the transition kernel of an irreducible Markov chain. Let $\hat P:= \sum_{i=0}^{m}p_i P^{i} $ where $\sum_{i=0}^m p_i=1$, $m \in \N$. If the greatest common denominator (gcd) of $\{i \in [m]:p_i>0 \} $ is 1 then $P$ is recurrent iff $\hat P$ is recurrent.    
\end{lemma*}
\begin{proof}
Clearly $\hat P$ is also irreducible. Let $\hat P^i$ be $(\hat P)^i$. Let $x$ be an arbitrary state. Clearly,  $\sum_{i \ge 0 } \hat P^i(x,x) \le (1-p_0)^{-1} \sum_{i \ge 0 }P^i(x,x) $. Conversely, consider the random walk $(S_t)_{t \in \Z_+}$ on $\Z_+$ started from $0$ whose transition kernel is given by $K(i,i+j)=p_j $. Let $q_{\ell} $ be the probability that $S_t=\ell$ for some $t$. It is standard that the assumption that the gcd of $\{i \in [m]:p_i>0 \} $ is 1 implies that there is some $\delta>0$ and $N_0 \in \N $ such that for all $\ell \ge N_0$ we have that $q_{\ell} \ge \delta$ (we shall use the lemma only in the case  when $p_i>0 $ for all $i \in [m] $, for which the proof of the last claim is particularly simple).  Finally, it is easy to see that $\sum_{i \ge 0 } \hat P^i(x,x) \ge \sum_{i \ge 0 } q_i P^i(x,x) \ge \delta \sum_{i \ge N_0 } P^i(x,x)$. Thus $\sum_{i \ge 0 } \hat P^i(x,x)  $ diverges iff $\sum_{i \ge 0 }  P^i(x,x)  $ diverges.
\end{proof}

\subsection{Proof of Lemma \ref{lem:Csym}}
\label{s:simplelem2}
\begin{lemma*}
In the setup of Definition \ref{d:alphah} we have that
 $\pi_{\mathcal{A}}$ is stationary for both $\mathbf{Y}^{\ee} $ and $\mathbf{Y}^{\vv} $. Moreover,   for all $a,b \in \mathcal{A} $ we have that
 \begin{equation*}
 P_{\mathcal{A},\ee}(a,b)=P_{\mathcal{A},\ee}(b^{\rr},a^{\rr}) \quad \text{and} \quad P_{\mathcal{A},\vv}(a,b)=P_{\mathcal{A},\vv}(b^{\rr},a^{\rr}).
 \end{equation*} 
\end{lemma*}
  
\begin{proof}
Let $-(g_1,\ldots,g_k):=(-g_1,\ldots,-g_k)$. Observe that $-\alpha_h=-\alpha_{-h-sk}^{\rr} $ and that  $-\alpha_h^{\rr}=\alpha_{-h-sk} $.   By translation $P_{\mathcal{A},\ee}(\alpha_a,\alpha_b)=P_{\mathcal{A},\ee}(\alpha_{a+c},\alpha_{b+c}) $, $P_{\mathcal{A},\ee}(\alpha_a^{\rr},\alpha_b^{\rr})=P_{\mathcal{A},\ee}(\alpha_{a+c}^{\rr},\alpha_{b+c}^{\rr}) $,  $P_{\mathcal{A},\ee}(\alpha_a,\alpha_b^{\rr})=P_{\mathcal{A},\ee}(\alpha_{a+c},\alpha_{b+c}^{\rr}) $ and  $P_{\mathcal{A},\ee}(\alpha_a^{\rr},\alpha_b)=P_{\mathcal{A},\ee}(\alpha_{a+c}^{\rr},\alpha_{b+c}) $,   for all $a,b,c \in H$. The same holds w.r.t.\ $P_{\mathcal{A},\vv} $. Also, $P_{\mathcal{A},\ee}(\alpha_x,\alpha_y)=P_{\mathcal{A},\ee}(-\alpha_x,-\alpha_y)$ and  $P_{\mathcal{A},\vv}(\alpha_x,\alpha_y)=P_{\mathcal{A},\vv}(-\alpha_x,-\alpha_y)$. To see this,  consider an arbitrary path $(\alpha_x=a_{0},a_1,\ldots,a_r=\alpha_y)$ with $a_i \in \Omega_k^{\ee} $ (respectively, $\in \Omega_k^{\vv} $)  for all $i$.  Then  $(-\alpha_x=-a_{0},-a_1,\ldots,-a_r=-\alpha_y)$ is a path from $-\alpha_x $ to $-\alpha_y $. Moreover, the probability the edge (respectively, vertex) $k$-NBRW started from $\alpha_x$ follows the first path equals the probability  the edge (respectively, vertex) $k$-NBRW started from $-\alpha_x$ follows the second path. Indeed, this follows from the fact that $x \to -x $ is an automorphism of $G$ which extends to a probability preserving bijection of $\Omega_k^{\ee} $ (respectively, $\Omega_k^{\vv}$) in the sense that $P_{k,\ee}(a,b)=P_{k,\ee}(-a,-b) $ for all $a,b \in \Omega_k^{\ee} $ (respectively,  $P_{k,\vv}(a,b)=P_{k,\vv}(-a,-b) $ for all $a,b \in \Omega_k^{\vv} $). 

Hence \[P_{\mathcal{A},\ee}(\alpha_a,\alpha_b)=P_{\mathcal{A},\ee}(\alpha_{-ks},\alpha_{b-a-ks})=P_{\mathcal{A},\ee}(-\alpha_{-ks},-\alpha_{b-a-ks})\] \[=P_{\mathcal{A},\ee}(\alpha_{0}^{\rr},\alpha_{a-b}^{\rr})=P_{\mathcal{A},\ee}(\alpha_{b}^{\rr},\alpha_{a}^{\rr})\]
(where 0 is the identity element of $H$) and similarly  $P_{\mathcal{A},\vv}(\alpha_a,\alpha_b)=P_{\mathcal{A},\vv}(\alpha_{b}^{\rr},\alpha_{a}^{\rr}) $. 

Similarly,    
 \[P_{\mathcal{A},\ee}(\alpha_a,\alpha_b^{\rr})=P_{\mathcal{A},\ee}(\alpha_{-ks},\alpha_{b-a-ks}^{\rr})=P_{\mathcal{A},\ee}(\alpha_{0}^{\rr},\alpha_{a-b})=P_{\mathcal{A},\ee}(\alpha_{b}^{\rr},\alpha_{a}).\]
The identities $P_{\mathcal{A},\ee}(\alpha_a^{\rr},\alpha_b)=P_{\mathcal{A},\ee}(\alpha_b^{\rr},\alpha_a) $, $P_{\mathcal{A},\vv}(\alpha_a,\alpha_b^{\rr})=P_{\mathcal{A},\vv}(\alpha_b,\alpha_a^{\rr})$ and $P_{\mathcal{A},\vv}(\alpha_a^{\rr},\alpha_b)=P_{\mathcal{A},\vv}(\alpha_b^{\rr},\alpha_a) $ are proved in the same manner.
\end{proof}

\end{document}